\documentclass{article}

\usepackage[margin=1in]{geometry}
\usepackage[numbers]{natbib}

\usepackage[utf8]{inputenc} 
\usepackage[T1]{fontenc}    
\usepackage[colorlinks,linkcolor=blue,filecolor=blue,citecolor=black,urlcolor=blue]{hyperref}       
\usepackage{url}    

\usepackage{booktabs}      
\usepackage{amsfonts}       
\usepackage{nicefrac}       
\usepackage{microtype}      
\usepackage[table]{xcolor}         
\usepackage{tocvsec2}
\usepackage{titletoc}

\usepackage{xspace}

\usepackage{amsmath}
\usepackage{amsthm}
\usepackage{amssymb}

\usepackage{multicol}
\usepackage{multirow}
\usepackage{makecell}
\usepackage{enumitem}
\usepackage{wrapfig}
\usepackage{dsfont}

\usepackage{graphicx,epstopdf}
\usepackage{subfigure} 

\usepackage{algorithm}
\usepackage{algorithmic}

\usepackage{threeparttable, booktabs}
\usepackage{mathtools}
\usepackage{float}

\usepackage{hyperref}[pdfusetitle]
\hypersetup{colorlinks,urlcolor=blue}
\usepackage{pifont}       % \ding{xx}
\usepackage{bbding}       % \Checkmark,\XSolid,... (需要和pifont宏包共同使用)

\newcommand{\cmark}{{\color{green} \ding{51}}}
\newcommand{\xmark}{{\color{red} \ding{55}}}

\usepackage{tabularray}

\usepackage[skins,listings]{tcolorbox}

\newtcblisting{tikzexample}{
    sidebyside,
    center lower,
    bicolor,
    colbacklower=white,
    sharp corners,
    frame engine=empty
}

\usepackage{quiver}

\restylefloat{algorithm} 
\newtheorem{assumption}{Assumption}
\newtheorem{theorem}{Theorem}
\newtheorem{remark}{Remark}
\newtheorem{lemma}{Lemma}
\newtheorem{example}{Example}

\newtheorem{corollary}{Corollary}

\usepackage{diagbox}

\def\r{{\boldsymbol{r}}}

\def\u{{\boldsymbol{u}}}

\def\w{\boldsymbol{w}}

\allowdisplaybreaks

\title{On the Convergence of Stochastic Gradient Descent with Perturbed Forward-Backward Passes}

\author{
  Boao Kong$^1$$^*$\\
  \texttt{kongboao@stu.pku.edu.cn} \\
   \and
  Hengrui Zhang$^2$$^*$ \\
  \texttt{2022141210007@stu.scu.edu.cn} \\
  \and
  Kun Yuan$^1$$^\dagger$ \\
  \texttt{kunyuan@pku.edu.cn} \\
}

\begin{document}
\setlength{\parindent}{0pt}
\setlength{\parskip}{0.5em}
\maketitle

\def\thefootnote{$^1$}\footnotetext{Peking University.}
\def\thefootnote{$^2$}\footnotetext{Sichuan University.}
\def\thefootnote{$^*$}\footnotetext{Equal contribution.}
\def\thefootnote{$^\dagger$}\footnotetext{Corresponding author.}

\begin{abstract}
We study stochastic gradient descent (SGD) for composite optimization problems with $N$ sequential operators subject to perturbations in both the forward and backward passes. Unlike classical analyses that treat gradient noise as additive and localized, perturbations to intermediate outputs and gradients cascade through the computational graph, compounding geometrically with the number of operators. We present the first comprehensive theoretical analysis of this setting. Specifically, we characterize how forward and backward perturbations propagate and amplify within a single gradient step, derive convergence guarantees for both general non-convex objectives and functions satisfying the Polyak--\L{}ojasiewicz condition, and identify conditions under which perturbations do not deteriorate the asymptotic convergence order. As a byproduct, our analysis furnishes a theoretical explanation for the gradient spiking phenomenon widely observed in deep learning, precisely characterizing the conditions under which training recovers from spikes or diverges. Experiments on logistic regression with convex and non-convex regularization validate our theories, illustrating the predicted spike behavior and the asymmetric sensitivity to forward versus backward perturbations.
\end{abstract}

{\small
\tableofcontents}

\allowdisplaybreaks

\section{Introduction}
This paper investigates the following optimization problem with a composite objective involving \( N \geq 2 \) operators: 
\begin{subequations}
\label{pipeline parallel}
\begin{align}
\min_{\w \in\mathbb{R}^d}&\quad \ell(\w):=\mathbb{E}_{x\in\mathcal{D}}\Big[\mathcal{L}(x;\w):=\mathcal{F}(x;\w)+\mathcal{R}(\w)\Big], \label{eq:pipeline-1}\\
\text{s.t.}&\quad \mathcal{F}(x;\w):=y_N,\quad y_i = f_i(y_{i-1}, w_i), \quad \forall i \in \{1, \ldots, N\}, \label{eq:pipeline-2}
\end{align}
\end{subequations}
where each operator $f_i: \mathbb{R}^{d_{i-1}} \times \mathbb{R}^{d_{w_i}} \to \mathbb{R}^{d_i}$ represents the $i$-th component of the composite function $\mathcal{F}(x; \w)$, with $y_i \in \mathbb{R}^{d_i}$ and $w_i \in \mathbb{R}^{d_{w_i}}$. The global optimization variable is defined as $\w := (w_1^{\top}, w_2^{\top}, \ldots, w_N^{\top})^{\top} \in \mathbb{R}^d$, where each $w_i$ corresponds to operator $f_i$, and the total dimension is $d = d_{w_1} + d_{w_2} + \cdots + d_{w_N}$. Here, the term $\mathcal{R}(\w)$ denotes a differentiable regularizer used to prevent model overfitting. Quantity $x$ denotes a data sample drawn from the distribution $\mathcal{D}$, and $y_i$ represents the intermediate output of the $i$-th operator $f_i$, with the initial condition $y_0 = x\in\mathbb{R}^{d_0}$. 

\subsection{Motivating Examples}
The general formulation~\eqref{pipeline parallel} arises in various scenarios involving sequential decision-making and operator-wise processing. Several motivating examples are provided below.

\vspace{1mm}
\textbf{Example 1: Deep neural network.} A deep neural network (DNN) is a prime example of the problem in \eqref{pipeline parallel}, where each operator $f_i$ represents a neural network layer ~\cite{carreira2014distributed}. Each layer transforms its input $y_{i-1}$ using learnable parameters $w_i$ to produce the output $y_i$. The composite function $\mathcal{F}(x; \w)$ consists of these successive transformations, with the final layer generating the value of the loss function. The hierarchical structure and nonlinearity of DNNs allow them to model complex data relationships~\cite{hornik1989multilayer}, making them particularly effective for tasks such as image recognition, natural language processing, and reinforcement learning ~\cite{he2016deep,vaswani2017attention,mnih2015human}.

\vspace{1mm}
\textbf{Example 2: Finite horizon optimization.} Traditional optimization seeks to ensure good performance as the number of iterations \( T \) approaches infinity~\cite{nesterov2018lectures}. In contrast, finite horizon optimization~\cite{zhang2024finite} is to identify optimal hyperparameter specifically for a small and finite $T$. For instance, to determine optimal learning rate schedules for gradient descent over $T$ iterations, we formulate the problem~\cite{drori2014performance}:
\begin{subequations}
\label{eq:fho}
\begin{align}
\min_{\{\gamma_1,\cdots,\gamma_{T}\}}&\quad \max_{g\in\mathcal{G}}\ \varepsilon(x_T),\\
\text{s.t.}\hspace{6mm} &\quad x_{t} = x_{t-1} - \gamma_{t} \nabla g(x_{t-1}), \quad \forall t \in \{1, \ldots, T\}.
\end{align}
\end{subequations}
Here, $x_T$ is obtained through $T$ gradient descent steps with learning rate $\gamma_t$ at iteration $t$. Function $\varepsilon(\cdot)$ evaluates the performance of $x_T$, and $\mathcal{G}$ denotes the objective function class. Problem~\eqref{eq:fho} seeks optimal $\{\gamma_1,\ldots,\gamma_{T}\}$ to achieve the best possible performance after $T$ iterations~\cite{gupta2023nonlinear}. By setting $\w=\{\gamma_1,\ldots,\gamma_{T}\}$, $y_{T+1}=\max_{g\in \mathcal{G}}\varepsilon(x_T)$, $y_t = x_t$ for $t\in\{1,\ldots, T\}$, $y_0 = x_0$, $f_t(y_{t-1},\gamma_t) = y_{t-1} - \gamma_{t} \nabla g(y_{t-1})$, and $\mathcal{F}(x;\w) = y_{T+1}$, problem~\eqref{eq:fho} reduces to the deterministic version of problem~\eqref{pipeline parallel}.

\vspace{1mm}
\textbf{Example 3: Linear quadratic control.} The linear quadratic control (LQC) problem~\cite{khalil1996robust, anderson2007optimal} addresses optimal control of linear systems. The goal is to find a control policy based on observed information that minimizes the expected value of a quadratic cost function over a finite time horizon. The problem is given as follows:
\begin{subequations}
\label{eq:lqg}
\begin{align}
    \min_{\pi} &\quad \mathbb{E}\left[\sum_{t=0}^{T-1}\left(x_t^{\top} Q_x x_t+u_t^{\top} R_u u_t\right)+x_T^{\top} Q_{T} x_T\right], \\
    \text{s.t.} &\quad u_t = \pi_t(x_{0:t}; u_{0:t-1}), \quad \hspace{1.5mm} \forall t \in \{0,\ldots,T-1\}, \\
    &\quad x_{t+1}=A x_t+B u_t, \quad \quad \forall t \in \{0,\ldots,T-1\}.
\end{align}
\end{subequations}
Here, the state $x_t$ evolves according to the linear combination of previous state $x_{t-1}$ and control input $u_{t-1}$ with coefficient matrices $A$ and $B$. The policy $\pi_t$ generates control action $u_t$ based on the history of states $x_{0:t}$. Problem~\eqref{eq:lqg} seeks the optimal policy $\pi = (\pi_1, \ldots, \pi_{T-1})$ to minimize the total expected cost, determined by weighting matrices $Q_x$, $R_u$, and $Q_T$. By setting $\w = \pi = (\pi_1, \ldots, \pi_{T-1})$, $y_t = [u_t;x_{t+1}]$, and $f_t(y_{t-1}, \pi_t) = [\pi_t(x_{0:t}; u_{0:t-1});A x_t+B u_t ]$, problem~\eqref{eq:lqg} reduces to~\eqref{pipeline parallel}.

\subsection{Stochastic Gradient Descent}
Stochastic Gradient Descent (SGD) is a widely used optimization algorithm for solving the problem described in~\eqref{pipeline parallel}. Given the sequential nature of the operators in this formulation, each iteration of SGD involves both forward and backward passes to compute the stochastic gradient. Let \( u_i := \nabla_{w_i} \mathcal{F}(x; \w) \) and \( v_i := \nabla_{y_i} \mathcal{F}(x; \w) \) denote the stochastic gradients with respect to weight \( w_i \) and intermediate output \( y_i \), respectively, where \( v_N = 1 \) (since the output \( y_N \) is the final operator's result). We also use $r_i:=\nabla_{w_i}\mathcal{R}(\w)$ to denote the gradient of $\mathcal{R}(\w)$ with respect to $w_i$. The SGD procedure for optimizing~\eqref{pipeline parallel} is as follows:
\begin{subequations}
\label{eq:SGD-FB}
\begin{alignat}{3}
% \label{eq:sgd_forward_wo_error}
&\text{(Forward)}\quad  &\quad &y_i^{(t)} = f_i(y_{i-1}^{(t)}, w_i^{(t)}),& &\quad i = 1, 2, \dots, N, \label{eq:fw}\\
&\text{(Backward)}\quad & &v_{i-1}^{(t)} = \nabla_1 f_i (y_{i-1}^{(t)}, w_i^{(t)})^{\top} v_i^{(t)},& &\quad  i = N, N-1, \dots, 2, \label{eq:bw} \\
&\text{(Gradient)}\quad & &u_i^{(t)} = \nabla_2 f_i(y_{i-1}^{(t)}, w_i^{(t)})^{\top} v_i^{(t)},& &\quad  i = N, N-1, \dots, 1, \label{eq:wg}\\
&\text{(Update)}\quad & &w_i^{(t+1)} = w_i^{(t)} - \gamma (u_i^{(t)}+r_i^{(t)}),& &\quad  i = N, N-1, \dots, 1. \label{eq:update}
\end{alignat}
\end{subequations}
We refer to $y_i$, $v_i$, $u_i$, and $w_i$ as the \textit{intermediate output}, \textit{intermediate gradient}, \textit{weight gradient}, and \textit{weight}, respectively. Algorithm~\eqref{eq:SGD-FB} is structured into four key stages:
\begin{enumerate}[left=1.5em]
\vspace{1mm}
\item \textbf{Forward pass.} The input \( x \) is propagated through all \( N \) operators using \eqref{eq:fw}, where \( y_0^{(t)} = x \). This step sequentially evaluates all operators.

\vspace{1mm}
\item \textbf{Backward pass.} Gradient information is propagated in reverse order. Intermediate gradient \( v_i^{(t)} \) is computed at each operator \( i \) starting from \( i = N \) down to \( i = 1 \). Specifically, \( v_{i-1}^{(t)} \) is updated by applying the chain rule \eqref{eq:bw}, where \( \nabla_1 f_i \in \mathbb{R}^{d_i \times d_{i-1}} \) is the Jacobian of the operator \( f_i \) with respect to \( y_{i-1} \).

\vspace{1mm}
\item \textbf{Gradient.} The parameter gradients \( u_i^{(t)} \) are computed at each operator by combining $v_i^{(t)}$ with \( \nabla_2 f_i \in \mathbb{R}^{d_i \times d_{w_i}} \), the local Jacobians of the operators with respect to the weight \( w_i \), see the update in \eqref{eq:wg}.

\vspace{1mm}
\item \textbf{Weight update.} The weights \( w_i \) are updated with \eqref{eq:update}, where \( \gamma \) is the learning rate, and the update proceeds from \( i = N \) down to \( i = 1 \).
\end{enumerate}

\vspace{1mm}
\noindent This iterative procedure minimizes the composite objective function \( \mathcal{F}(x;\w) \) together with the regularizer \(\mathcal{R}(\w)\) by updating the operator-wise parameters \(\{w_i\}_{i=1}^N\) using gradient information with respect to both intermediate outputs and model weights. Note that \(\mathcal{R}(\w)\) depends only on \(\w\), so its gradient \(r_i=\nabla_{w_i}\mathcal{R}(\w)\) can be computed and added directly in the update step, without participating in the forward/backward passes for evaluating \(\mathcal{F}(x;\w)\).

\subsection{SGD with Perturbed Forward and Backward Passes} 
In practical implementations, gradient evaluation is susceptible to perturbations in both the forward and backward passes. To model these effects, we introduce forward and backward perturbation terms, ${\delta_i^{(t)}}$ and ${\varepsilon_i^{(t)}}$, for the \(i\)-th operator at iteration \(t\). We model perturbations only in the forward/backward evaluation of the sample-dependent composite loss \(\mathcal{F}(x;\w)\); the parameter-only regularizer \(\mathcal{R}(\w)\) is treated as exact and contributes deterministically via \(r_i=\nabla_{w_i}\mathcal{R}(\w)\). Incorporating these perturbations leads to an SGD update with perturbed forward and backward passes:
\begin{subequations}
\label{eq: sgd_forward_w_error}
\begin{alignat}{3}
&\text{(Forward)}&\quad &\tilde{y}_i^{(t)} = f_i(\tilde{y}_{i-1}^{(t)},w_i^{(t)})+{\delta_i^{(t)}},& &i=1,2,\cdots,N, \label{eq:forward-purt}\\
&\text{(Backward)}\quad& &\tilde{v}_{i-1}^{(t)} = \nabla_1 f_i (\tilde{y}_{i-1}^{(t)},w_i^{(t)})^{\top}\tilde{v}_i^{(t)}+{\varepsilon_i^{(t)}},&\quad & i=N,N-1,\cdots,2, \label{eq:back-purt}\\
&\text{(Gradient)}\quad& &\tilde{u}_{i}^{(t)}= \nabla_2 f_i(\tilde{y}_{i-1}^{(t)},w_i^{(t)})^{\top}\tilde{v}_i^{(t)},& &i=N,N-1,\cdots,1, \label{eq:gradient-purt} \\
&\text{(Update)}\quad & &w_i^{(t+1)} = w_i^{(t)} - \gamma (\tilde{u}_i^{(t)}+r_i^{(t)}),& &i = N, N-1, \dots, 1. \label{eq:update-purt}
\end{alignat}    
\end{subequations}
Here, $\tilde{y}_i$ and $\tilde{v}_i$ denote the corrupted intermediate output and intermediate gradient, induced by (random) perturbations $\delta_i$ and $\varepsilon_i$, respectively. Since the weight gradient $\tilde{u}_i$ depends on the corrupted intermediate gradient $\tilde{v}_i$ in~\eqref{eq:gradient-purt}, it yields a corrupted gradient $\tilde{u}_i$. Figure \ref{fig: perturbed_sgd} illustrates algorithm procedure listed in \eqref{eq: sgd_forward_w_error}.

\begin{figure}[t]
    \centering
\includegraphics[width=0.92\linewidth]{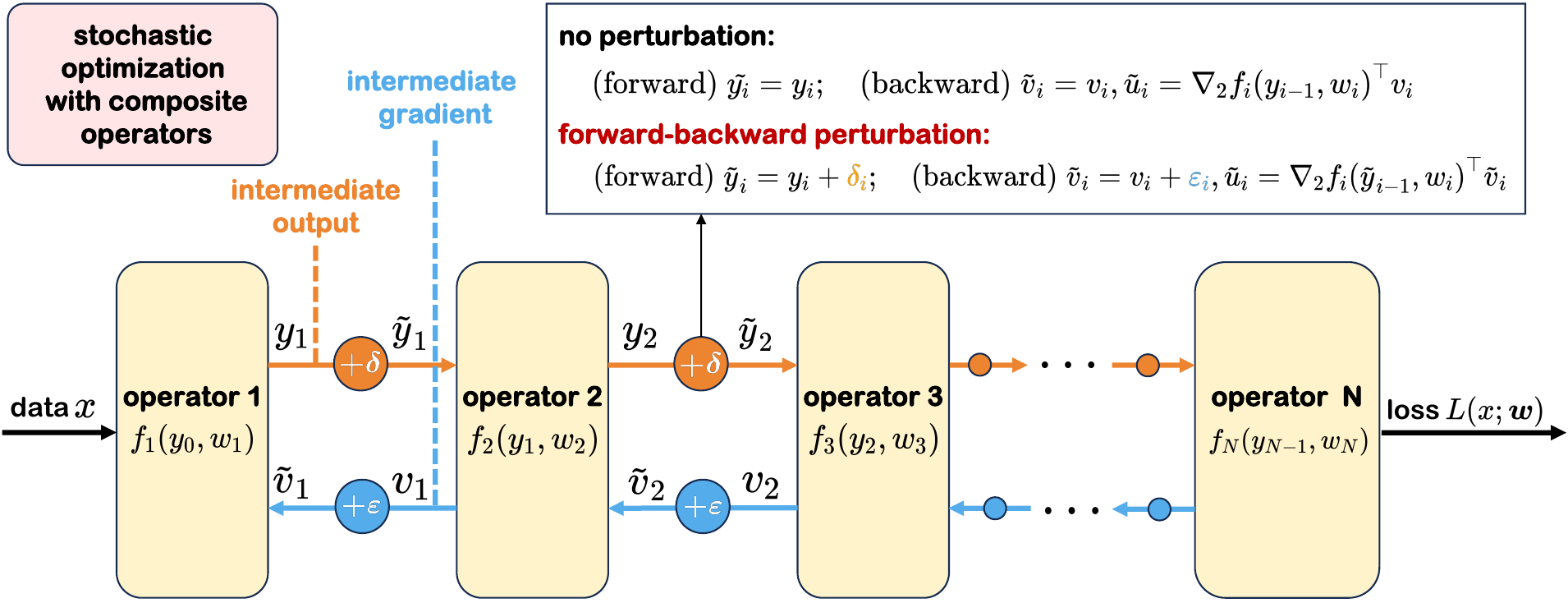}
    \caption{An illustration of SGD algorithm with perturbed forward and backward passes.}
    \label{fig: perturbed_sgd}
\end{figure}

The perturbed SGD framework~\eqref{eq: sgd_forward_w_error} naturally emerges in numerous practical scenarios. For instance, mixed-precision training~\cite{micikevicius2017mixed,bondarenko2021understanding,micikevicius2022fp8,fishman2024scaling,chen2024int,tailor2020degree,chen2024efficientqat}, widely adopted for DNN, combines low-precision computations (FP16/INT8) with FP32 model weights to reduce memory consumption and accelerate training on specialized hardware. This precision reduction inherently introduces systematic perturbations: forward passes in lower precision generate rounding errors ($\delta_i$), while reduced-precision intermediate gradient computations yield backward perturbations ($\varepsilon_i$). Similarly, pipeline-parallel DNN training partitions neural network layers across multiple devices, requiring communication of intermediate outputs and gradients. To mitigate communication overhead, these variables undergo compression before transmission~\cite{evans2020jpeg,fu2020don,wang2022fine,kong2025clapping}, inducing forward and backward perturbations at operator boundaries; see Appendix~\ref{section: A3} for a simple two-layer example where a deterministic Top-1 compressor prevents convergence even though the uncompressed method converges. Additionally, hardware-induced errors—such as bit flips from cosmic rays or voltage instability~\cite{rakin2019bit,guerrero2022reliability,dong2023one}—introduce unpredictable computational perturbations during forward and backward passes.
Furthermore, in the LQC problem~\eqref{eq:lqg}, the state \( x_t \)  may be affected by random Gaussian noise, e.g., \( x_{t+1} = A x_t + B u_t + \delta_t \), where \( \delta_t \sim \mathcal{N}(0, \Sigma_\delta) \). These diverse sources of perturbations demonstrate that SGD with perturbed forward and backward passes captures the fundamental dynamics of modern optimization, where exact gradients may not always be feasible.

\subsection{Open Questions, Theoretical Challenges, and Contributions} 

\subsubsection{Open questions} Despite the widespread use of SGD with perturbed forward and backward passes in modern optimization and machine learning, our theoretical understanding remains limited. Most existing literature focuses exclusively on noise or perturbations to weight gradients $u_i$ in~\eqref{eq:gradient-purt}, where errors remain localized—they neither propagate forward through subsequent operators nor backward through intermediate gradients, affecting only the model weight update. In contrast, perturbations to intermediate outputs $y_i$ in~\eqref{eq:forward-purt} and intermediate gradients $v_i$ in~\eqref{eq:back-purt} cascade through the computational graph, accumulating across operators to create compounded errors that fundamentally complicate convergence analysis. This paper investigates the following three fundamental open questions: 
\begin{itemize}[left=1.5em]
\vspace{2mm}
\item[\textbf{Q1}.] \textbf{(Error dynamics)} How can the propagation and accumulation of compounded errors through the forward and backward passes be theoretically characterized? 

\vspace{1mm}
\item[\textbf{Q2.}] \textbf{(Convergence)} Can SGD converge to optimal solutions under perturbed forward and backward passes, and how do these perturbations quantitatively impact the convergence behavior?

\vspace{1mm}
\item[\textbf{Q3.}] \textbf{(Robustness)} Under what conditions do perturbations not affect convergence rates, and what  properties of perturbations contribute to this robustness?
\end{itemize}

\subsubsection{Theoretical challenges} 
The main challenge arises from compounded error propagation through computational graphs. Standard SGD analyses~\cite{bottou2018optimization} treat gradient errors as additive noise that is independent of computation—modeling perturbed gradients as $\tilde{u} = \nabla f(w) + \varepsilon$ with exogenous $\varepsilon$. In contrast, framework~\eqref{eq: sgd_forward_w_error} generates perturbations endogenously that cascade through operators, creating complex inter-operator dependencies. Consider a linear neural network with $f_i(y_{i-1},W_i)=W_iy_{i-1}$, zero forward perturbations ($\delta_i^{(t)}=0$), and constant backward perturbations ($\varepsilon_i^{(t)}=\varepsilon_i$). The weight gradient $\tilde{u}_1$ under both models reveals the distinction:
\begin{subequations}
\label{error propagation}
\begin{alignat}{2}
\label{error propagation_standard}
&\text{(Add. perturbation)}& &\tilde{u}_1 =  W_2^\top \cdots W_N^\top v_N \cdot v_1^\top+\varepsilon,\\
\label{error propagation_F and B}
&\text{(Acc. perturbation)}&\quad &\tilde{u}_1 =  \left( W_2^\top \left( \cdots \left( W_N^\top v_N + \varepsilon_N \right) + \cdots \right) + \varepsilon_2 \right)v_1^\top+ \varepsilon_1. 
\end{alignat}    
\end{subequations}
Here, $\tilde{u}_1$ denotes the perturbed gradient with respect to $W_1$, while ``Add.'' and ``Acc.'' denote additive and accumulative perturbations, respectively. The perturbed formulation~\eqref{error propagation_F and B} exhibits three pathological behaviors absent from standard SGD analyses with additive noise: (1) \textbf{Geometric error amplification.} The deepest perturbation $\varepsilon_N$ is multiplied by the matrix product $W_1^\top \cdots W_{N-1}^\top$, potentially causing exponential growth with network depth—a phenomenon that can trigger convergence failure. (2) \textbf{Cascading error interference.} The nested structure creates nonlinear interactions between operator-wise perturbations, where errors compound in ways that simple noise models cannot capture.
(3) \textbf{Biased gradient estimates.} Even when perturbations $\delta_i^{(t)}$ and $\varepsilon_i^{(t)}$ have zero expectation, the nonlinear operators in~\eqref{eq: sgd_forward_w_error} induce bias such that $\mathbb{E}[\nabla_2 f_i (\tilde{y}_{i-1}^{(t)},w_i^{(t)})] \neq \nabla_2 f_i ({y}_{i-1}^{(t)},w_i^{(t)})$.
These effects critically challenge  analysis, requiring new theoretical frameworks beyond classical SGD theory.

\subsubsection{Contributions} 
This paper provides the first comprehensive analysis for SGD under perturbed forward and backward passes. Our contributions are:
\begin{itemize}[left=1.5em]
\vspace{2mm}
\item[\textbf{C1.}] We theoretically characterize the propagation of forward and backward perturbations within a single gradient step, demonstrating that these perturbations amplify exponentially with the number of operators, which addresses {Q1.} 

\vspace{1mm}
\item[\textbf{C2.}] We establish convergence guarantees for perturbed SGD  on both general non-convex functions and those satisfying the Polyak–Łojasiewicz (PL) condition. By explicitly quantifying the impact of forward and backward errors on the convergence rate, we address open question {Q2}. 

\vspace{1mm}
\item[\textbf{C3.}] We identify the specific conditions under which forward and backward perturbations do not deteriorate the asymptotic order of the convergence rate, directly addressing open question Q3. 
\end{itemize}

\vspace{1mm}
\noindent As a side result, our analysis provides an explanation for the ``gradient spiking'' phenomenon commonly observed in deep learning. These spikes are sudden, temporary jumps in the gradient, often triggered by perturbations during the forward and backward passes. While some spikes are recoverable, meaning they can be corrected during the training process, others can lead to training divergence, see Figure~\ref{fig: spike_phenomenon} for the illustration. Our theoretical analysis of the perturbed SGD~\eqref{eq: sgd_forward_w_error} clearly establishes the conditions under which the optimization can recover or experience a complete crash.

\begin{figure}
	\centering
\includegraphics[width=0.92\linewidth]{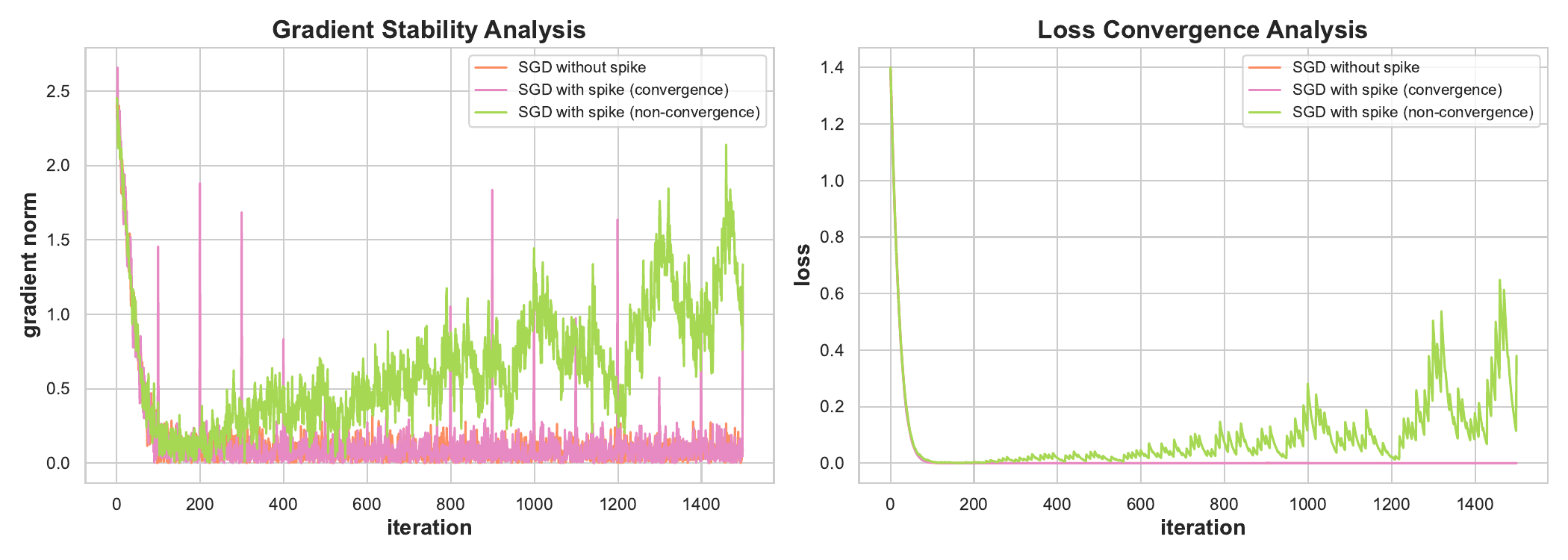}
	\caption{Comparison of gradient spike patterns leading to contrasting outcomes. The gradient norm (left) and loss (right) trajectories illustrate two scenarios: a large-magnitude spike (orange) that permits rapid recovery and convergence, versus a more moderate spike (blue) that triggers persistent deviation and eventual non-convergence.}
	\label{fig: spike_phenomenon}
\end{figure}

\vspace{1mm}
\textbf{Organization.} This paper is organized as follows: Section \ref{section: Related works} presents  related work. Section \ref{section: Preliminaries} presents preliminaries, including notations and assumptions. Section \ref{section: Error propagation analysis} analyzes the propagation of forward and backward computation errors. Section \ref{section: Convergence rate with computation error} derives the convergence rate of the SGD algorithm with computation error under non-smooth settings and PL conditions. Section \ref{section: Recover from computation error and gradient spike} establishes conditions on computation error that guarantee convergence and enable recovery from gradient spikes. Section \ref{section: Experiments} presents numerical experiments validating our theoretical findings. Finally, Section \ref{section: Conclusion} concludes the paper.

\section{Related Works}
\label{section: Related works}
Our analysis builds upon and extends several lines of theoretical work on optimization algorithms subject to computational inaccuracies.

\textbf{Algorithms with inexact gradient oracles.} The effect of computational inaccuracy on gradient methods has been studied extensively under inexact first-order oracle models, in which the algorithm accesses a perturbed gradient at each iteration. Classical results analyze gradient descent with deterministic errors (e.g., finite differences in zeroth-order methods, early termination of inner solvers, discretization errors in PDE-constrained optimization, and surrogate model approximations) and derive convergence and complexity guarantees that degrade gracefully with the error level~\cite{bertsekas2000gradient,devolder2014first,d2008smooth}. In stochastic settings, the gradient is typically corrupted by additive sampling noise from random data sampling; related analyses quantify the convergence performance typically under conditional moment or bounded-variance assumptions \cite{ghadimi2013stochastic,shamir2013stochastic,bottou2018optimization}. A parallel line of work examines inexact proximal point and proximal-gradient methods, in which proximal subproblems are solved approximately while still ensuring descent and convergence \cite{humes2005inexact,burachik2010inexact,schmidt2011convergence,hamadouche2021approximate}. More recent developments refine oracle models for composite structures and higher-order inexactness, and establish non-asymptotic bounds that explicitly track oracle accuracy \cite{sun2020composite,alkousa2024higher,nabou2025proximal}.

Many pre-training approaches for large language models can be viewed as first-order methods with inexact gradients due to communication or computation errors. Examples include gradient quantization and sparsification \cite{alistarh2017qsgd,richtarik2021ef21,fatkhullin2023momentum}, mixed-precision training \cite{micikevicius2017mixed,micikevicius2022fp8,liu2024deepseek}, and gradient clipping~\cite{koloskova2023revisiting,qian2021understanding}. The analysis typically treats the gradient as a perturbation of the ideal stochastic gradient and relates convergence to error statistics (bias, variance, or relative bounds). This perspective is particularly powerful when the perturbation can be modeled as an \emph{additive error} to the computed gradient, as illustrated in \eqref{error propagation_standard}, enabling standard descent and martingale analyses.

The perturbation mechanism in \eqref{eq: sgd_forward_w_error} differs fundamentally from standard inexact-oracle models. Perturbations are injected \emph{within} the forward–backward pass at intermediate activations and adjoints, propagating through the computational graph before the gradient forms. The resulting error is endogenous, shaped by local Jacobians at the current iterate, creating inter-operator dependencies. Moreover, even zero-mean perturbations can induce gradient bias due to the composite structure; see Appendix~\ref{section: example_zero-mean noise} for a simple scalar example. These considerations motivate our analysis of SGD with perturbed forward–backward passes.

\vspace{1mm}
\textbf{Stochastic compositional optimization (SCO).}
Stochastic compositional optimization \cite{wang2017accelerating,wang2017stochastic,chen2021solving} studies problems with nested expectation structure, such as
\begin{equation}
\min_{w\in\mathbb{R}^d} F(w):= \mathbb{E}_{\zeta}\bigl[f_{\zeta}\bigl( G(w) \bigr)\bigr],\qquad G(w) := \mathbb{E}_{\xi}[g_{\xi}(w)],
\label{eq:sco-two-level-shortest}
\end{equation}
and their multi-level generalizations \cite{yang2019multilevel,balasubramanian2022stochastic,zhangxiao2021nestedvr,zhanglan2021optimal,jiang2022optimal,jiang2025revisiting} (see also constrained or projection-free variants \cite{xiao2022projectionfree,jiang2024pmvr}). A key difficulty is that computing $\nabla F(w)$ requires evaluating the derivative at $G(w)$, an inner expectation that is not observable from a single sample $\xi$. Practical methods thus maintain auxiliary variables (typically one per level) to track the nested expectations and form quasi-gradients by combining component derivative estimates with these tracked quantities. Since the outer derivative employs an evolving approximation instead of the exact inner expectations, the descent directions become biased. Consequently, the analysis focuses on controlling the coupled iterate-tracking dynamics through multi-timescale step sizes or single-loop tracking/variance-reduction designs.

Our setting differs from SCO in both formulation and error mechanism. At a high level, both SCO and our framework involve compositional graphs where intermediate quantities are accessed inexactly, leading to biased gradient surrogates. However, the source and nature of inexactness are fundamentally different. In SCO, inexactness arises from \emph{estimating inner expectations} within each operator; this bias can be reduced algorithmically by allocating more computational effort to inner estimation or tracking---for example, using larger inner batches or more accurate variance-reduced estimators. In contrast, our setting~\eqref{eq: sgd_forward_w_error} involves a single expectation, where inexactness stems from perturbations \emph{within} the forward and backward passes themselves. Crucially, increasing batch size reduces sampling variance but \emph{does not} diminish these within-pass perturbations, which persist independently of sample size. Moreover, perturbations in our setting are {depth-coupled}: they propagate and accumulate along sample paths through the computational graph. This central challenge---depth-wise propagation and accumulation of intermediate disturbances---lies beyond the scope of existing SCO analyses \cite{wang2017accelerating,wang2017stochastic,chen2021solving,yang2019multilevel,balasubramanian2022stochastic}.

\vspace{1mm}
\textbf{AQ-SGD. }The most closely related work is AQ-SGD, which analyzes activation quantization for distributed pipeline training \cite{wang2022fine}. Critically, AQ-SGD is tailored to pipeline parallelism with communication-induced perturbations and does not address general forward-backward perturbation mechanisms or provide the depth-wise error propagation and disturbance-frequency analysis developed here.

\section{Preliminaries} 
\label{section: Preliminaries}
This section introduces necessary notations and assumptions.
\subsection{Notations} 
We denote the gradient of $\ell(\w)$ by $\nabla \ell(\w)\in \mathbb{R}^d$. For $i = 1, \ldots, N$, the partial derivative of $\ell(\w)$ with respect to parameter $w_i$ is denoted by $\nabla_i \ell(\w):= \frac{\partial \ell(\w)}{\partial w_i} \in \mathbb{R}^{d_{w_i}}$. For each operator $f_i(y_{i-1},w_i): \mathbb{R}^{d_{i-1}} \times \mathbb{R}^{d_{w_i}} \to \mathbb{R}^{d_i}$, the Jacobian matrices with respect to the input $y_{i-1}$ and parameter $w_i$ are denoted by $\nabla_{1}f_i(y_{i-1}, w_i) \in \mathbb{R}^{d_i \times d_{i-1}}$ and $\nabla_{2}f_i(y_{i-1}, w_i) \in \mathbb{R}^{d_i \times d_{w_i}}$, respectively. We also denote $\nabla^2_{21}f_i(y_{i-1}, w_i)\in \mathbb{R}^{d_i \times d_{w_i}\times d_{i-1}}$ as the partial derivative of $\nabla_{2}f_i(y_{i-1}, w_i)$ with respect to $y_{i-1}$. We use $\|\cdot\|$ to denote the Euclidean norm for vectors and Frobenius norm for matrices and tensors, and $\|\cdot\|_{\text{op}}$ to denote the induced operator norm. Throughout the paper, superscripts $t$ indicate the iteration index, while subscripts $i$ denote the component index. For any family of component-wise vectors $\{a_i\}_{i=1}^N$ (e.g., $u_i$, $v_i$, $r_i$ and their perturbed counterparts), we use boldface $\boldsymbol{a}:=(a_1^\top,\ldots,a_N^\top)^\top$ to denote the stacked/concatenated vector across components
(and similarly $\tilde{\mathbf{a}}$ for the perturbed version). We define the minimum objective value as $\ell^*:=\min_{\w}\ell(\w)$ and the initial optimality gap as $\Delta_0:=\ell(\w^{(0)})-\ell^*$. The notation $a \lesssim b$ means $a \leq Cb$ for some constant $C > 0$ independent of problem parameters. Notation $\tilde{O}(\cdot)$ hides logarithmic dependence on $T$.

For the stochastic setting, let $\mathcal{F}^{(t)}$ denote the filtration generated up to iteration $t$, with $\mathbb{E}_t[\cdot]$ denoting the conditional expectation given $\mathcal{F}^{(t)}$. Furthermore, let $\mathcal{G}_i^{(t)} = \sigma(\mathcal{F}^{(t)} \cup \{\tilde{y}_1^{(t)}, \ldots, \tilde{y}_i^{(t)}\})$ denote the filtration generated by perturbed outputs up to operator $i$ in the forward pass, and let $\mathcal{H}_i^{(t)} = \sigma(\mathcal{G}_N^{(t)} \cup \{\tilde{v}_N^{(t)}, \ldots, \tilde{v}_i^{(t)}\})$ denote the filtration generated by the full forward pass and backward gradients from operator $N$ down to operator $i$. We use $\mathbb{E}^{\mathcal{G}_i}_t[\cdot]$ and $\mathbb{E}^{\mathcal{H}_i}_t[\cdot]$ to denote conditional expectations with respect to $\mathcal{G}_i^{(t)}$ and $\mathcal{H}_i^{(t)}$, respectively.

\subsection{Third-Order Tensor and Its Operator Norm}
We define multiplication between third-order tensors and matrices, along with the induced operator norm.

Let $\mathcal{H} := (\mathcal{H}_{j\alpha\beta}) \in \mathbb{R}^{p \times m \times n}$ be a third-order tensor and $A := (A_{\alpha\beta}) \in \mathbb{R}^{m \times n}$ be a matrix. The product $\mathcal{H}(A) \in \mathbb{R}^p$ is given by contracting over the last two indices:

\begin{equation}
  [\mathcal{H}(A)]_j
  :=
  \sum_{\alpha=1}^m \sum_{\beta=1}^n
    \mathcal{H}_{j\alpha\beta} A_{\alpha\beta},
  \qquad j = 1, \dots, p.
  \label{eq:tensor-matrix-product}
\end{equation}
That is, for each $j$, the component $[\mathcal{H}(A)]_j$ equals $\sum_{\alpha,\beta} \mathcal{H}_{j\alpha\beta} A_{\alpha\beta}$. From \eqref{eq:tensor-matrix-product}, we define the operator norm of $\mathcal{H}$ as
\begin{equation}
  \|\mathcal{H}\|_{\mathrm{op}}
  :=
  \sup_{A \in \mathbb{R}^{m \times n}, \, A \neq 0}
    \frac{\|\mathcal{H}(A)\|_2}{\|A\|}
  =
  \sup_{\|A\| = 1} \|\mathcal{H}(A)\|_2.
  \label{eq:tensor-op-norm}
\end{equation}
Intuitively, $\|\mathcal{H}\|_{\mathrm{op}}$ measures the largest amplification of the output norm $\|\mathcal{H}(A)\|_2$ when the input matrix $A$ has unit Frobenius norm. The Frobenius norm of $\mathcal{H}$ can be expressed as
\begin{equation}
\|\mathcal{H}\|:=\Big(\sum_{j,\alpha,\beta}\mathcal{H}_{j,\alpha,\beta}^2\Big)^{\frac{1}{2}}.
  \label{eq:tensor-f-norm}
\end{equation}

\subsection{Assumptions} 
The following two assumptions are standard in the convergence analysis of stochastic gradient algorithms.

%------------------------------------------------
% A1
%------------------------------------------------

\begin{assumption}[\sc Smoothness of the loss function]
\label{assumption:smoothness}
    There exists a constant $L_{\nabla \ell} > 0$ such that $\nabla \ell(\w)$ is $L_{\nabla \ell}$-Lipschitz continuous.
\end{assumption}

\begin{assumption}[\sc Stochasticity]
\label{assumption:unbiased}
    There exists a constant $\sigma>0$ such that for any given $x\in\mathbb{R}^{d_{0}}$, the unperturbed gradient oracles satisfy:
    \begin{align*}
        \mathbb{E}_{x\in\mathcal{D}}[\nabla_{\w} \mathcal{L}(x;\w)] = \nabla\ell(\w);\quad \mathbb{E}_{x\in\mathcal{D}}\left[\left\Vert\nabla_{\w} \mathcal{L}(x;\w)-\nabla\ell(\w)\right\Vert^2\right] \leq\sigma^2.
    \end{align*}
\end{assumption}

%------------------------------------------------
% A2
%------------------------------------------------
We also make the following assumption on the smoothness of each operator.
\begin{assumption}[\sc Operator smoothness]
\label{assumption:smoothness of the components}
For each $i=1,\ldots,N$, there exist constants $C_{\nabla f},C_{\nabla^2f},L_f,L_{\nabla f},L_{\nabla^2f}>0$ such that for all $y_{i-1}\in\mathbb{R}^{d_{i-1}}$ and $w_i\in\mathbb{R}^{d_{w_i}}$:
\begin{enumerate}[left=1.5em]
\vspace{1mm}
    \item The Jacobians satisfy
    $$\|\nabla_j f_i(y_{i-1},w_i)\|_{\text{op}}\leq C_{\nabla f}\quad(j=1,2), \qquad\|\nabla^2_{21}f_i(y_{i-1},w_i)\|_{\text{op}}\leq C_{\nabla^2f}.$$
    \item The operator $f_i$ is Lipschitz continuous with respect to $y_{i-1}$ with constant $L_f$. Here, Lipschitz continuity of $f_i$ is measured in the Euclidean norm.
    \item The mappings $\nabla f_{i}$, and $\nabla^2 f_{i}$ are Lipschitz continuous with respect to both $y_{i-1}$ and $w_i$, with constants $L_{\nabla f}$, and $L_{\nabla^2 f}$, respectively. Here, Lipschitz continuity of $\nabla f_{i}$ and $\nabla^2 f_{i}$ is measured in the Frobenius norm.
\end{enumerate}
\end{assumption}

Assumption~\ref{assumption:smoothness of the components} is an operator-regularity condition that controls how perturbations amplify through the computational graph during the forward and backward passes. Bounded-Jacobian and Lipschitz-type conditions are standard in analyses involving intermediate perturbations; see, e.g.,~\cite{chen2021actnn,wang2022fine,kong2025clapping,liu2022gact,castiglia2022compressed}.\footnote{When the component mappings are twice differentiable, the Lipschitz requirements on the derivative maps in Assumption~\ref{assumption:smoothness of the components} can be interpreted as bounded variation of the local Jacobians, or equivalently, bounded mixed second-order derivatives.} 

We emphasize that Assumption~\ref{assumption:smoothness of the components} is \emph{not} an algorithmic constraint: throughout the paper we analyze the plain SGD update in~\eqref{eq:update-purt} without projection or clipping. That said, Assumption~\ref{assumption:smoothness of the components} can be interpreted \emph{locally}: it suffices for the stated regularity bounds to hold on a compact region containing the iterates and intermediate states visited during training. Such compactness can arise, for example, from finite data combined with bounded parameter regimes, explicit projection to a compact set, or standard clipping/normalization practices. Our goal is to isolate the effect of forward/backward perturbations; incorporating such safeguards would only strengthen stability in practice.
Finally, we make the following assumption of the Polyak-Łojasiewicz (PL) condition \cite{polyak1963gradient} for $\ell(\w)$:

%------------------------------------------------
% APL
%------------------------------------------------
\begin{assumption}[PL condition]
\label{assumption:pl}
    There exist a constant $\mu>0$ such that for any $\w\in\mathbb{R}^d$, it holds that
    \begin{align}
\left\Vert\nabla\ell(\w)\right\Vert^2\geq2\mu(\ell(\w)-\ell^*).
    \end{align}
\end{assumption}

The PL condition can be implied by strong convexity. In particular, adding an $\ell_2$ regularizer $\mathcal{R}(\w)=\frac{\lambda}{2}\|\w\|^2$ makes the loss $\ell(\w)$ $\lambda$-strongly convex whenever the data-fitting term is convex, which in turn ensures the PL inequality holds with $\mu=\lambda$.

\vspace{3mm}

\section{Error Propagation Analysis}
\label{section: Error propagation analysis}
This section characterizes the propagation of forward and backward perturbations within a single gradient step. We begin with the following useful lemma.

\begin{lemma}
\label{lemma:boundness_v}
Under Assumption~\ref{assumption:smoothness of the components}, there exists a constant $C_v^2 \geq 1$ such that $\Vert v_i^{(t)}\Vert^2 \leq C_v^2$ for all $t = 1, 2, \ldots, T$ and $i = 1, 2, \ldots, N$.
\end{lemma}

\begin{proof}
For $i = 1, 2, \ldots, N-1$, the vector $v_i^{(t)}$ can be expressed as
\begin{align}
    v_i^{(t)} &= \left(\nabla_1 f_{N}(y_{N-1}^{(t)}, w_{N}^{(t)})\right)^{\top} \cdots \left(\nabla_1 f_{i+2}(y_{i+1}^{(t)}, w_{i+2}^{(t)})\right)^{\top} \cdot \left(\nabla_1 f_{i+1}(y_i^{(t)}, w_{i+1}^{(t)})\right)^{\top}. \nonumber 
\end{align}
This yields the upper bound
\begin{equation}
\begin{aligned}
    \left\Vert v_i^{(t)}\right\Vert^2 &\leq \prod_{k=i+1}^N \left\Vert \nabla_1 f_{k}(y_{k-1}^{(t)}, w_{k}^{(t)})\right\Vert_{\text{op}}^2 \leq C_{\nabla f}^{2(N-i)},
\end{aligned}
\end{equation}
where the first inequality holds because $\nabla_1 f_{N}(y_{N-1}^{(t)}, w_{N}^{(t)})$ has dimension $d_{N-1}$, so its operator norm equals its Euclidean norm. Since $N$ is a fixed constant and $v_N^{(t)} = 1$, taking $C_v = \max\{1, C_{\nabla f}^{N}\}$ ensures that $\Vert v_i^{(t)}\Vert^2 \leq C_v^2$ for all $t = 1, 2, \ldots, T$ and $i = 1, 2, \ldots, N$.
\end{proof}

Lemma~\ref{lemma:boundness_v} ensures that the intermediate backward gradients $v_i^{(t)}$ remain uniformly bounded across all operators and iterations. This controls how perturbations introduced during the backward pass are amplified as they propagate toward earlier operators. The bound $C_v = \max\{1, C_{\nabla f}^N\}$ reflects worst-case accumulation of Jacobian norms across the network depth.
The following theorem presents an upper bound of the $\ell_2$-error term:
\begin{theorem}\label{thm-utilde-u-0}
    Suppose Assumption \ref{assumption:smoothness of the components} holds, then for $i=1,2,\cdots,N-1$, there exist constants $C^{e}_{\delta_i},C^{e}_{\varepsilon_{i+1}}\geq0$, which are defined in \eqref{expr-C-delta-epsilon}, such that:
    \begin{align}
    \label{error_variance}
    \mathbb{E}\big[ \| \tilde{\u}^{(t)}-{\u}^{(t)} \|^2 \big]
    \leq &\sum_{i=1}^{N-1}C^{e}_{\delta_i}\mathbb{E}\big[  \| \delta_{i}^{(t)}  \|^2 \big]+\sum_{i=1}^{N-1}C^{e}_{\varepsilon_{i+1}}\mathbb{E}\big[  \| \varepsilon_{i+1}^{(t)}  \|^2 \big],   
    \end{align}
    holds for all $t=1,2,\cdots,T$.
\begin{proof}
We first consider the gradient evaluation at the last component. We have
\begin{equation}
\label{estimation_u_1}
\begin{aligned}
    &\mathbb{E}\left[ \left\| \tilde{u}_N^{(t)}-{u}_N^{(t)} \right\|^2 \right]
    = \mathbb{E}\left[ \left\| \nabla_{2}f_N(\tilde{y}_{N-1}^{(t)},w_N^{(t)})  - \nabla_{2}f_N(y_{N-1}^{(t)},w_{N}^{(t)}) \right\|^2 \right] \\
    \leq &L_{\nabla f}^2 \mathbb{E}\left[  \left\| \tilde{y}_{N-1}^{(t)}-y_{N-1}^{(t)} \right\|^2 \right]\leq 2C_v^2L_{\nabla^2 f}^2 \mathbb{E}\left[  \left\| \tilde{y}_{N-1}^{(t)}-y_{N-1}^{(t)} \right\|^2 \right],   
\end{aligned}
\end{equation}
where the inequality follows from the smoothness of $\nabla_2 f_N$ and the fact that $C_v^2 \geq 1$.
For $i = 1, 2, \ldots, N-1$, we have
\begin{equation}
\label{estimation_u_2}
\begin{aligned}
&\mathbb{E}\left[ \left\| \tilde{u}_i^{(t)}-{u}_i^{(t)} \right\|^2 \right] = \mathbb{E}\left[ \left\| \nabla_{2} f_i (\tilde{y}_{i-1}^{(t)},w_i^{(t)})^{\top} \tilde{v}_i^{(t)}-\nabla_{2} f_i (y_{i-1}^{(t)},w_i^{(t)})^{\top}v_i^{(t)} \right\|^2 \right] \\
\leq &2\mathbb{E}\left[  \left\| \nabla_{2} f_i (\tilde{y}_{i-1}^{(t)},w_i^{(t)})\right\|_{\text{op}}^2\left\|\tilde{v}_i^{(t)}-v_i^{(t)} \right\|^2 \right]\\
&+2\mathbb{E}\left[  \left\| \nabla_{2} f_i (\tilde{y}_{i-1}^{(t)},w_i^{(t)}) -\nabla_{2} f_i (y_{i-1}^{(t)},w_i^{(t)})\right\|_{\text{op}}^2\left\|v_i^{(t)}  \right\|^2 \right]\\
\leq &2C_{\nabla f}^2 \mathbb{E}\left[  \left\| \tilde{v}_i^{(t)}-v_i^{(t)} \right\|^2 \right]+2C_{v}^2L_{\nabla f}^2\mathbb{E}\left[  \left\| \tilde{y}_{i-1}^{(t)}-y_{i-1}^{(t)}  \right\|^2 \right],  
\end{aligned}
\end{equation}
where the second inequality follows from the smoothness of $\nabla_2 f_i$ together with the boundedness of $v_i^{(t)}$ and $\nabla_2 f_i$. Summing over $i = 1, 2, \ldots, N$ yields
\begin{equation}
\label{estimation_u_3}
\begin{aligned}
&\mathbb{E}\left[ \left\| \tilde{\u}^{(t)}-{\u}^{(t)} \right\|^2 \right]=\sum_{i=1}^N\mathbb{E}\left[ \left\| \tilde{u}_i^{(t)}-{u}_i^{(t)} \right\|^2 \right]\\
\leq &2C_{\nabla f}^2\sum_{i=1}^{N-1}\mathbb{E}\left[  \left\| \tilde{v}_i^{(t)}-v_i^{(t)} \right\|^2 \right]+2C_{v}^2L_{\nabla f}^2\sum_{i=1}^{N-1}\mathbb{E}\left[  \left\| \tilde{y}_{i}^{(t)}-y_{i}^{(t)}  \right\|^2 \right].    
\end{aligned}
\end{equation}
In view of~\eqref{estimation_u_3}, we next analyze the error between $\tilde{v}_i^{(t)}$ and $v_i^{(t)}$ for $i = 1, 2, \ldots, N-1$. We obtain
\begin{equation}
\label{estimation_v_1}
\begin{aligned}
&\mathbb{E}\hspace{-0.5mm}\left[\hspace{-0.5mm}\left\| \tilde{v}_i^{(t)}-{v}_i^{(t)} \right\|^2\hspace{-0.5mm}\right]
\hspace{-1mm}=\hspace{-1mm}\mathbb{E}\hspace{-0.5mm}\left[\hspace{-0.5mm}\left\| \nabla_{1} f_{i+1} (\tilde{y}_{i}^{(t)},w_{i+1}^{(t)})^{\top} \tilde{v}_{i+1}^{(t)}\hspace{-0.5mm}-\hspace{-0.5mm}\nabla_{1} f_{i+1} (y_{i}^{(t)},w_{i+1}^{(t)})^{\top} v_{i+1}^{(t)}\hspace{-0.5mm}+\hspace{-0.5mm}\varepsilon_{i+1}^{(t)} \right\|^2\hspace{-0.5mm}\right] \\
\leq &3C_{\nabla f}^2 \mathbb{E}\left[  \left\| \tilde{v}_{i+1}^{(t)}-v_{i+1}^{(t)} \right\|^2 \right]+3C_{v}^2L_{\nabla f}^2\mathbb{E}\left[  \left\| \tilde{y}_{i}^{(t)}-y_{i}^{(t)}  \right\|^2 \right]+3\mathbb{E}\left[  \left\| \varepsilon_{i+1}^{(t)}  \right\|^2 \right]\\
\leq&\cdots\leq3C_{v}^2L_{\nabla f}^2\sum_{j=i}^{N-1}(3C_{\nabla f}^2)^{j-i}\mathbb{E}\left[  \left\| \tilde{y}_{j}^{(t)}-y_{j}^{(t)}  \right\|^2 \right]+3\sum_{j=i}^{N-1}(3C_{\nabla f}^2)^{j-i}\mathbb{E}\left[  \left\| \varepsilon_{j+1}^{(t)}  \right\|^2 \right].
\end{aligned}
\end{equation}
Finally, we bound the error between $\tilde{y}_i^{(t)}$ and $y_i^{(t)}$ for $i = 1, 2, \ldots, N-1$:
\begin{equation}
\label{estimation_y_1}
\begin{aligned}
&\mathbb{E}\left[ \left\| \tilde{y}_i^{(t)}-{y}_i^{(t)} \right\|^2 \right]
= \mathbb{E}\left[ \left\| f_i(\tilde{y}_{i-1}^{(t)},w_i^{(t)})-f_i(y_{i-1}^{(t)},w_i^{(t)})+\delta_i^{(t)} \right\|^2 \right] \\
\leq &2L_{f}^2 \mathbb{E}\left[  \left\| \tilde{y}_{i-1}^{(t)}-{y}_{i-1}^{(t)}  \right\|^2 \right]+2\mathbb{E}\left[  \left\| \delta_{i}^{(t)}  \right\|^2 \right]
\leq\cdots\leq2\sum_{j=1}^i(2L_f^2)^{i-j}\mathbb{E}\left[  \left\| \delta_{j}^{(t)}  \right\|^2 \right].    
\end{aligned}
\end{equation}
Substituting~\eqref{estimation_v_1} into~\eqref{estimation_u_3}, we obtain
\begin{equation}
\label{estimation_u_4}
\begin{aligned}
&\mathbb{E}\left[ \left\| \tilde{\u}^{(t)}-{\u}^{(t)} \right\|^2 \right]\\
\leq &6C_{\nabla f}^2C_{v}^2L_{\nabla f}^2\sum_{i=1}^{N-1}\sum_{j=i}^{N-1}(3C_{\nabla f}^2)^{j-i}\mathbb{E}\left[  \left\| \tilde{y}_{j}^{(t)}-y_{j}^{(t)}  \right\|^2 \right]\\
&+2C_{v}^2L_{\nabla f}^2\sum_{i=1}^{N-1}\mathbb{E}\left[  \left\| \tilde{y}_{i}^{(t)}-y_{i}^{(t)}  \right\|^2 \right]+6C_{\nabla f}^2\sum_{i=1}^{N-1}\sum_{j=i}^{N-1}(3C_{\nabla f}^2)^{j-i}\mathbb{E}\left[  \left\| \varepsilon_{j+1}^{(t)}  \right\|^2 \right]\\
= &2C_{v}^2L_{\nabla f}^2\sum_{i=1}^{N-1}\left[\sum_{j=0}^{i}(3C_{\nabla f}^2)^{j}\right]\mathbb{E}\left[  \left\| \tilde{y}_{i}^{(t)}-y_{i}^{(t)}  \right\|^2 \right]+2\sum_{i=1}^{N-1}\left[\sum_{j=1}^{i}(3C_{\nabla f}^2)^{j}\right]\mathbb{E}\left[  \left\| \varepsilon_{i+1}^{(t)}  \right\|^2 \right].    
\end{aligned}
\end{equation}
Substituting~\eqref{estimation_y_1} into~\eqref{estimation_u_4} then gives
\begin{equation}
\label{estimation_u_5}
\begin{aligned}
\mathbb{E}\left[ \left\| \tilde{\u}^{(t)}-{\u}^{(t)} \right\|^2 \right]
\leq &4C_{v}^2L_{\nabla f}^2\sum_{i=1}^{N-1}\left(\sum_{j=i}^{N-1}\left[\sum_{k=0}^{j}(3C_{\nabla f}^2)^{k}\right](2L_f^2)^{j-i}\right)\mathbb{E}\left[  \left\| \delta_{i}^{(t)}  \right\|^2 \right]\\
&+\sum_{i=1}^{N-1}\left[\sum_{j=1}^{i}(3C_{\nabla f}^2)^{j}\right]\mathbb{E}\left[  \left\| \varepsilon_{i+1}^{(t)}  \right\|^2 \right].    
\end{aligned}
\end{equation}
Defining 
\begin{align}
\label{expr-C-delta-epsilon}
C^{e}_{\delta_i}:=4C_{v}^2L_{\nabla f}^2\sum_{j=i}^{N-1}\left[\sum_{k=0}^{j}(3C_{\nabla f}^2)^{k}\right](2L_f^2)^{j-i},\quad C^{e}_{\varepsilon_{i+1}}:=\sum_{j=1}^{i}(3C_{\nabla f}^2)^{j},
\end{align}
we conclude that~\eqref{error_variance} holds for all $t = 1, 2, \ldots, T$.
\end{proof}
\end{theorem}

Inequality~\eqref{error_variance} shows that the $\ell_2$ error in the stochastic gradient due to forward and backward perturbations is controlled by the $\ell_2$ magnitudes of $\delta$ and $\varepsilon$. Crucially, the coefficients $C^{e}_{\delta_i}$ and $C^{e}_{\varepsilon_{i+1}}$ scale exponentially with the depth of the computational graph, reflecting $e^{\mathcal{O}(N)}$ error amplification. This exponential dependence illustrates how small perturbations at individual operators can compound into significant gradient distortions as they propagate through the network.

We next bound the second moment of the bias with the following lemma.
\begin{lemma}
\label{lemma:biased term estimation}
Under Assumption \ref{assumption:smoothness of the components}, it holds for $i=1,2,\cdots,N$ that:
\begin{equation}
\begin{aligned}
\label{eq:lemma_second moment}
    &\left\Vert\mathbb{E}^{\mathcal{G}_{i-1}}_t[\nabla_{2} f_i (\tilde{y}_{i-1}^{(t)},w_i^{(t)}) -\nabla_{2} f_i (y_{i-1}^{(t)},w_i^{(t)})]\right\Vert^2\\
    \leq&2C_{\nabla^2 f}^2\left\Vert\mathbb{E}^{\mathcal{G}_{i-1}}_t[\tilde{y}_{i-1}^{(t)}-{y}_{i-1}^{(t)}]\right\Vert^2+L_{\nabla^2 f}^2\mathbb{E}^{\mathcal{G}_{i-1}}_t\left[\left\Vert\tilde{y}_{i-1}^{(t)}-{y}_{i-1}^{(t)}\right\Vert^4\right].
\end{aligned}
\end{equation}

\begin{proof}
    Since $f$ is differentiable, we have
    \begin{equation*}
    \begin{aligned}
    &\left\Vert\mathbb{E}^{\mathcal{G}_{i-1}}_t[\nabla_{2} f_i (\tilde{y}_{i-1}^{(t)},w_i^{(t)}) -\nabla_{2} f_i (y_{i-1}^{(t)},w_i^{(t)})]\right\Vert^2\\  
    =&\left\Vert\mathbb{E}^{\mathcal{G}_{i-1}}_t\left[\int_0^1\nabla_{21} f_i (y_{i-1}^{(t)}+t(\tilde{y}_{i-1}^{(t)}-{y}_{i-1}^{(t)}),w_i^{(t)})\left(\tilde{y}_{i-1}^{(t)}-{y}_{i-1}^{(t)}\right) dt\right]\right\Vert^2\\
    \leq&2\left\Vert\nabla_{21} f_i (y_{i-1}^{(t)},w_i^{(t)})\mathbb{E}^{\mathcal{G}_{i-1}}_t[\tilde{y}_{i-1}^{(t)}-{y}_{i-1}^{(t)}]\right\Vert^2\\
    &\hspace{-1.5mm}+\hspace{-0.8mm}2\hspace{-0.5mm}\left\Vert\mathbb{E}^{\mathcal{G}_{i\hspace{-0.3mm}-\hspace{-0.3mm}1}}_t\hspace{-1.5mm}\left[\hspace{-0.5mm}\int_0^1\hspace{-1.8mm}\left(\hspace{-0.5mm}\nabla_{21} f_i (y_{i-1}^{(t)} \hspace{-0.5mm} + \hspace{-0.5mm} t(\tilde{y}_{i-1}^{(t)} \hspace{-0.5mm} - \hspace{-0.5mm} {y}_{i-1}^{(t)}),\hspace{-0.3mm}w_i^{(t)}) \hspace{-0.5mm} - \hspace{-0.5mm}  \nabla_{21} f_i (y_{i-1}^{(t)},w_i^{(t)})\hspace{-0.5mm}\right)\hspace{-0.5mm}(\tilde{y}_{i-1}^{(t)} \hspace{-0.5mm} - \hspace{-0.5mm} {y}_{i-1}^{(t)}) dt\right]\hspace{-0.5mm}\right\Vert^2\\
    \leq&2C_{\nabla^2 f}^2\left\Vert\mathbb{E}^{\mathcal{G}_{i-1}}_t[\tilde{y}_{i-1}^{(t)}-{y}_{i-1}^{(t)}]\right\Vert^2+L_{\nabla^2 f}^2\mathbb{E}^{\mathcal{G}_{i-1}}_t\left[\left\Vert\tilde{y}_{i-1}^{(t)}-{y}_{i-1}^{(t)}\right\Vert^4\right],
    \end{aligned}
    \end{equation*}
    which completes the proof.
\end{proof}
\end{lemma}

Lemma~\ref{lemma:biased term estimation} bounds the second moment of the bias introduced by forward perturbations. Specifically, the squared deviation of the conditional expected gradient is controlled by two terms: the squared norm of the expected forward error $\mathbb{E}^{\mathcal{G}_{i-1}}_t[\tilde{y}_{i-1}^{(t)} - y_{i-1}^{(t)}]$ and the fourth moment of the forward error $\tilde{y}_{i-1}^{(t)} - y_{i-1}^{(t)}$. Building on this result, we derive an upper bound for the second moment of the overall bias as follows.

\begin{theorem}\label{thm-utilde-u}
    Suppose Assumption~\ref{assumption:smoothness of the components} holds. Then for $i = 1, 2, \ldots, N-1$, there exist constants $C^{b}_{\delta_i}, \tilde{C}^{b}_{\delta_i}, C^{b}_{\varepsilon_{i+1}} \geq 0$, which are defined in \eqref{eqwna}, such that
    \begin{equation}
    \begin{aligned}
    \label{error_bias}
&\mathbb{E}\left[\left\|\mathbb{E}_t[\tilde{\boldsymbol{u}}^{(t)}-\boldsymbol{u}^{(t)}]\right\|^2\right]\\
    \leq &\sum_{i=1}^{N-1}C_{\delta_i}^b\mathbb{E}\left[  \left\| \mathbb{E}^{\mathcal{G}_{i}}_t[\delta_{i}^{(t)}]  \right\|^2 \right]+\sum_{i=1}^{N-1}\tilde{C}_{\delta_i}^b\mathbb{E}\left[  \left\| \delta_{i}^{(t)}  \right\|^4 \right] +\sum_{i=1}^{N-1}C^{b}_{\varepsilon_{i+1}}\mathbb{E}\left[  \left\|\mathbb{E}_t^{\mathcal{H}_{i+1}}[\varepsilon_{i+1}^{(t)}]   \right\|^2 \right]
    \end{aligned}
    \end{equation}
    holds for all $t = 1, 2, \ldots, T$.
\begin{proof}
We first consider the gradient evaluation at the last component. By~\eqref{eq:lemma_second moment}, we have
\begin{equation}
\label{estimation_u_1_1}
\begin{aligned}
& \mathbb{E}\left[\left\| \mathbb{E}_t[\tilde{u}_N^{(t)}-{u}_N^{(t)} ]\right\|^2\right] =\mathbb{E}\left[\left\| \mathbb{E}_t[\nabla_{2}f_N(\tilde{y}_{N-1}^{(t)},w_N^{(t)})  - \nabla_{2}f_N(y_{N-1}^{(t)},w_{N}^{(t)}) ] \right\|^2\right] \\
\leq &4C_v^2C_{\nabla^2f}^2\mathbb{E}\left[\left\Vert\mathbb{E}^{\mathcal{G}_{N-1}}_t[\tilde{y}_{N-1}^{(t)}-y_{N-1}^{(t)}]\right\Vert^2\right]+2C_v^2L_{\nabla^2 f}^2 \mathbb{E}\left[  \left\| \tilde{y}_{N-1}^{(t)}-y_{N-1}^{(t)} \right\|^4 \right].    
\end{aligned}
\end{equation}
For $i = 1, 2, \ldots, N-1$, we have
\begin{equation}
\label{estimation_u_2_1_0}
\begin{aligned}
&\mathbb{E}\left[ \left\| \mathbb{E}_t[\tilde{u}_i^{(t)}-{u}_i^{(t)}] \right\|^2 \right]
=\mathbb{E}\left[ \left\|\mathbb{E}_t\left[ \nabla_{2} f_i (\tilde{y}_{i-1}^{(t)},w_i^{(t)})^{\top} \tilde{v}_i^{(t)}-\nabla_{2} f_i (y_{i-1}^{(t)},w_i^{(t)})^{\top}v_i^{(t)}\right] \right\|^2 \right] \\
\leq &2\mathbb{E}\left[ \left\|\mathbb{E}_t\left[ \nabla_{2} f_i (\tilde{y}_{i-1}^{(t)},w_i^{(t)})^{\top}\left(\tilde{v}_i^{(t)}-v_i^{(t)}\right)\right] \right\|^2 \right]\\
&+2\mathbb{E}\left[ \left\|\mathbb{E}_t\left[ \nabla_{2} f_i (\tilde{y}_{i-1}^{(t)},w_i^{(t)})-\nabla_{2} f_i (y_{i-1}^{(t)},w_i^{(t)})\right]^{\top}v_i^{(t)} \right\|^2 \right].
\end{aligned}
\end{equation}
Using the boundedness of $v_i^{(t)}$ and the fact that
\begin{align*}
   &\left\|\mathbb{E}_t\left[ \nabla_{2} f_i (\tilde{y}_{i-1}^{(t)},w_i^{(t)})^{\top}\left(\tilde{v}_i^{(t)}-v_i^{(t)}\right)\right] \right\|^2\hspace{-1.5mm}=\hspace{-1.0mm}\left\|\mathbb{E}_t\left[ \nabla_{2} f_i (\tilde{y}_{i-1}^{(t)},w_i^{(t)})^{\top}\mathbb{E}_t^{\mathcal{H}_i}[\tilde{v}_i^{(t)}-v_i^{(t)}]\right] \right\|^2\\
   \leq&\mathbb{E}_t\hspace{-1.0mm}\left[ \left\|\nabla_{2} f_i (\tilde{y}_{i-1}^{(t)},w_i^{(t)})\hspace{-0.5mm}^{\top}\hspace{-0.5mm}\mathbb{E}_t^{\mathcal{H}_i}[\tilde{v}_i^{(t)}\hspace{-0.5mm}-\hspace{-0.5mm}v_i^{(t)}] \right\|^2\right]\hspace{-1.5mm}=\hspace{-1.0mm}\mathbb{E}_t\hspace{-0.5mm}\left[\hspace{-0.5mm} \left\|\nabla_{2} f_i (\tilde{y}_{i-1}^{(t)},w_i^{(t)})\right\|^2\hspace{-0.5mm}\left\|\mathbb{E}_t^{\mathcal{H}_i}[\tilde{v}_i^{(t)}\hspace{-0.5mm}-\hspace{-0.5mm}v_i^{(t)}] \right\|^2\hspace{-0.5mm}\right]
\end{align*}
(which follows from $\mathcal{F}^{(t)} \subset \mathcal{H}_i^{(t)}$), we obtain
\begin{equation}
\label{estimation_u_2_1}
\begin{aligned}
&\mathbb{E}\left[ \left\| \mathbb{E}_t[\tilde{u}_i^{(t)}-{u}_i^{(t)}] \right\|^2 \right] \\
\leq &2\mathbb{E}\left[ \left\|\nabla_{2} f_i (\tilde{y}_{i-1}^{(t)},w_i^{(t)}) \right\|^2\left\|\mathbb{E}^{\mathcal{H}_i}_t[ \tilde{v}_i^{(t)}-v_i^{(t)}] \right\|^2 \right]\\
&+2C_v^2\mathbb{E}\left[ \left\|\mathbb{E}_t\left[ \nabla_{2} f_i (\tilde{y}_{i-1}^{(t)},w_i^{(t)})-\nabla_{2} f_i (y_{i-1}^{(t)},w_i^{(t)})\right] \right\|^2 \right]\\
\leq&2C_{\nabla f}^2\mathbb{E}\left[ \left\|\mathbb{E}^{\mathcal{H}_i}_t[ \tilde{v}_i^{(t)}-v_i^{(t)}] \right\|^2 \right]+4C_v^2C_{\nabla^2f}^2\mathbb{E}\left[\left\Vert\mathbb{E}^{\mathcal{G}_{i-1}}_t[\tilde{y}_{i-1}^{(t)}-y_{i-1}^{(t)}]\right\Vert^2\right]\\
&+2C_v^2L_{\nabla^2 f}^2 \mathbb{E}\left[  \left\| \tilde{y}_{i-1}^{(t)}-y_{i-1}^{(t)} \right\|^4 \right].    
\end{aligned}
\end{equation}
where the last inequality follows from the boundedness of $\nabla f_i$, Lemma~\ref{lemma:biased term estimation}, and the inclusion $\mathcal{F}^{(t)} \subset \mathcal{G}_{i-1}^{(t)}$. Summing over $i = 1, 2, \ldots, N$ yields
\begin{equation}
\label{estimation_u_3_1}
\begin{aligned}
&\mathbb{E}\left[ \left\| \mathbb{E}_t[\tilde{\u}^{(t)}-{\u}^{(t)}] \right\|^2 \right]=\sum_{i=1}^N\mathbb{E}\left[ \left\| \mathbb{E}_t[\tilde{u}_i^{(t)}-{u}_i^{(t)}] \right\|^2 \right]\\
\leq &2C_{\nabla f}^2\sum_{i=1}^{N-1}\mathbb{E}\left[ \left\|\mathbb{E}^{\mathcal{H}_i}_t[ \tilde{v}_i^{(t)}-v_i^{(t)}] \right\|^2 \right]+4C_v^2C_{\nabla^2f}^2\sum_{i=1}^{N-1}\mathbb{E}\left[\left\Vert\mathbb{E}^{\mathcal{G}_{i}}_t[\tilde{y}_{i}^{(t)}-y_{i}^{(t)}]\right\Vert^2\right]\\
&+2C_{v}^2L_{\nabla^2 f}^2\sum_{i=1}^{N-1}\mathbb{E}\left[  \left\| \tilde{y}_{i}^{(t)}-y_{i}^{(t)}  \right\|^4 \right]. 
\end{aligned}
\end{equation}
In view of~\eqref{estimation_u_1_1} and~\eqref{estimation_u_2_1}, we next analyze $\mathbb{E}_t^{\mathcal{H}_i}[\tilde{v}_i^{(t)} - v_i^{(t)}]$ for $i = 1, 2, \ldots, N-1$. We obtain
\begin{equation}
\label{estimation_v_1_1}
\begin{aligned}
&\mathbb{E}\left[ \left\| \mathbb{E}_t^{\mathcal{H}_i}[\tilde{v}_i^{(t)}-{v}_i^{(t)}] \right\|^2 \right]\\
\leq &\mathbb{E}\hspace{-0.5mm}\left[ \left\| \mathbb{E}_t^{\mathcal{H}_{i+1}}\hspace{-1.0mm}\left[\nabla_{1} f_{i+1} (\tilde{y}_{i}^{(t)},w_{i+1}^{(t)})^{\top} \tilde{v}_{i+1}^{(t)}\hspace{-0.5mm}-\hspace{-0.5mm}\nabla_{1} f_{i+1} (y_{i}^{(t)},w_{i+1}^{(t)})^{\top} v_{i+1}^{(t)}+\varepsilon_{i+1}^{(t)}\right]\hspace{-0.5mm} \right\|^2 \right] \\
\leq &\dfrac{3}{2}\mathbb{E}\left[ \left\| \mathbb{E}_t^{\mathcal{H}_{i+1}}\left[\nabla_{1} f_{i+1} (\tilde{y}_{i}^{(t)},w_{i+1}^{(t)})^{\top} \tilde{v}_{i+1}^{(t)}-\nabla_{1} f_{i+1} (y_{i}^{(t)},w_{i+1}^{(t)})^{\top} v_{i+1}^{(t)}\right] \right\|^2 \right]\\
&+3\mathbb{E}\left[ \left\| \mathbb{E}_t^{\mathcal{H}_{i+1}}\left[\varepsilon_{i+1}^{(t)}\right] \right\|^2 \right] \\
\leq&3C_{\nabla f}^2\mathbb{E}\left[ \left\|\mathbb{E}^{\mathcal{H}_{i+1}}_t[ \tilde{v}_{i+1}^{(t)}-v_{i+1}^{(t)}] \right\|^2 \right]+6C_v^2C_{\nabla^2f}^2\mathbb{E}\left[\left\Vert\mathbb{E}^{\mathcal{G}_{i}}_t[\tilde{y}_{i}^{(t)}-y_{i}^{(t)}]\right\Vert^2\right]\\
&+3C_v^2L_{\nabla^2 f}^2 \mathbb{E}\left[  \left\| \tilde{y}_{i-1}^{(t)}-y_{i-1}^{(t)} \right\|^4 \right]+3\mathbb{E}\left[ \left\| \mathbb{E}_t^{\mathcal{H}_{i+1}}[\varepsilon_{i+1}^{(t)}] \right\|^2 \right]\\
\leq&\cdots\leq6C_{v}^2C_{\nabla^2f}^2\sum_{j=i}^{N-1}(3C_{\nabla f}^2)^{j-i}\mathbb{E}\left[  \left\| \mathbb{E}^{\mathcal{G}_{j}}_t[\tilde{y}_{j}^{(t)}-y_{j}^{(t)}]  \right\|^2 \right]\\
&+\hspace{-0.5mm}3\hspace{-1.0mm}\sum_{j=i}^{N-1}\hspace{-0.5mm}(3C_{\nabla f}^2)^{j-i}\mathbb{E}\hspace{-0.5mm}\left[  \left\|\mathbb{E}_t^{\mathcal{H}_{j+1}}[\varepsilon_{j+1}^{(t)}]   \right\|^2 \right]\hspace{-1.0mm}+\hspace{-0.5mm}3C_{v}^2L_{\nabla^2 f}^2\hspace{-1.0mm}\sum_{j=i}^{N-1}\hspace{-0.5mm}(3C_{\nabla f}^2)^{j-i}\mathbb{E}\hspace{-0.5mm}\left[  \left\| \tilde{y}_{j}^{(t)}\hspace{-1.0mm}-\hspace{-0.5mm}y_{j}^{(t)}  \right\|^4\hspace{-0.5mm} \right]\hspace{-1.0mm}.
\end{aligned}
\end{equation}
Next, we bound $\mathbb{E}^{\mathcal{G}_{i}}_t[\tilde{y}_{i}^{(t)} - y_{i}^{(t)}]$ for $i = 2, \ldots, N-1$:
\begin{equation}
\label{estimation_y_1_1}
\begin{aligned}
&\mathbb{E}\left[ \left\| \mathbb{E}^{\mathcal{G}_{i}}_t[\tilde{y}_{i}^{(t)}-y_{i}^{(t)}] \right\|^2 \right]
= \mathbb{E}\left[ \left\| \mathbb{E}^{\mathcal{G}_{i}}_t\left[f_i(\tilde{y}_{i-1}^{(t)},w_i^{(t)})-f_i(y_{i-1}^{(t)},w_i^{(t)})+\delta_i^{(t)}\right] \right\|^2 \right] \\
\leq &2 \mathbb{E}\left[  \left\| \mathbb{E}^{\mathcal{G}_{i}}_t\left[f_i(\tilde{y}_{i-1}^{(t)},w_i^{(t)})-f_i(y_{i-1}^{(t)},w_i^{(t)})\right]  \right\|^2 \right]+2\mathbb{E}\left[  \left\| \mathbb{E}^{\mathcal{G}_{i}}_t[\delta_{i}^{(t)}]  \right\|^2 \right]\\
\leq&4C_{\nabla f}^2\left\Vert\mathbb{E}_t^{\mathcal{G}_{i-1}}[\tilde{y}_{i-1}^{(t)}-y_{i-1}^{(t)}]\right\Vert^2+2L_{\nabla f}^2\mathbb{E}\left[\left\Vert\tilde{y}_{i-1}^{(t)}-y_{i-1}^{(t)}\right\Vert^4\right]+2\mathbb{E}\left[  \left\| \mathbb{E}^{\mathcal{G}_{i}}_t[\delta_{i}^{(t)}]  \right\|^2 \right]\\
\leq&2L_{\nabla f}^2\sum_{j=1}^{i-1}(4C_{\nabla f}^2)^{i-j-1}\mathbb{E}\left[\left\Vert\tilde{y}_{j}^{(t)}-y_{j}^{(t)}\right\Vert^4\right]+2\sum_{j=1}^i(4C_{\nabla f}^2)^{i-j}\mathbb{E}\left[  \left\| \mathbb{E}^{\mathcal{G}_{j}}_t[\delta_{j}^{(t)}]  \right\|^2 \right],    
\end{aligned}
\end{equation}
where the second inequality follows from Lemma~\ref{lemma:biased term estimation} and the inclusion $\mathcal{G}_{i-1}^{(t)} \subset \mathcal{G}_{i}^{(t)}$. Note that~\eqref{estimation_y_1_1} also holds for $i = 1$. Finally, for $i = 1, 2, \ldots, N-1$, we have
\begin{equation}
\label{estimation_y_11_1}
\begin{aligned}
&\mathbb{E}\left[ \left\| \tilde{y}_i^{(t)}-{y}_i^{(t)} \right\|^4 \right]
= \mathbb{E}\left[ \left\| f_i(\tilde{y}_{i-1}^{(t)},w_i^{(t)})-f_i(y_{i-1}^{(t)},w_i^{(t)})+\delta_i^{(t)} \right\|^4 \right] \\
\leq &8L_{f}^4 \mathbb{E}\left[  \left\| \tilde{y}_{i-1}^{(t)}-{y}_{i-1}^{(t)}  \right\|^4 \right]+8\mathbb{E}\left[  \left\| \delta_{i}^{(t)}  \right\|^4 \right]
\leq\cdots\leq8\sum_{j=1}^i(8L_f^4)^{i-j}\mathbb{E}\left[  \left\| \delta_{j}^{(t)}  \right\|^4 \right].    
\end{aligned}
\end{equation}
Substituting~\eqref{estimation_v_1_1} into~\eqref{estimation_u_3_1} yields
\begin{equation}
\label{estimation_u_4_1}
\begin{aligned}
&\mathbb{E}\left[ \left\| \mathbb{E}_t[\tilde{\u}^{(t)}-{\u}^{(t)}] \right\|^2 \right]\\
\leq &2C_{v}^2\sum_{i=1}^{N-1}\left[\sum_{j=0}^{i}(3C_{\nabla f}^2)^{j}\right]\mathbb{E}\left[2C_{\nabla^2 f}^2\left\Vert\mathbb{E}^{\mathcal{G}_{i}}_t[\tilde{y}_{i}^{(t)}-y_{i}^{(t)}]\right\Vert^2+L_{\nabla^2 f}^2\left\| \tilde{y}_{i}^{(t)}-y_{i}^{(t)}  \right\|^4\right]\\
&+2\sum_{i=1}^{N-1}\left[\sum_{j=1}^{i}(3C_{\nabla f}^2)^{j}\right]\mathbb{E}\left[  \left\|\mathbb{E}_t^{\mathcal{H}_{i+1}}[\varepsilon_{i+1}^{(t)}]   \right\|^2 \right].  
\end{aligned}
\end{equation}
We now consider the first term on the right-hand side of~\eqref{estimation_u_4_1}. Using~\eqref{estimation_y_1_1}, we obtain
\begin{equation}
\begin{aligned}
    &4C_{\nabla^2 f}^2C_{v}^2\sum_{i=1}^{N-1}\left[\sum_{j=0}^{i}(3C_{\nabla f}^2)^{j}\right]\mathbb{E}\left[\left\Vert\mathbb{E}^{\mathcal{G}_{i}}_t[\tilde{y}_{i}^{(t)}-y_{i}^{(t)}]\right\Vert^2\right]\\
    \leq&8C_{\nabla^2 f}^2C_{v}^2\sum_{i=1}^{N-1}\left(\sum_{j=i}^{N-1}\left[\sum_{k=0}^{j}(3C_{\nabla f}^2)^{k}\right](4C_{\nabla f}^2)^{j-i}\right)\mathbb{E}\left[  \left\| \mathbb{E}^{\mathcal{G}_{i}}_t[\delta_{i}^{(t)}]  \right\|^2 \right]\\
    &+\hspace{-0.5mm}8C_{\nabla^2 f}^2L_{\nabla f}^2C_{v}^2\hspace{-0.5mm}\sum_{i=1}^{N-2}\hspace{-1.0mm}\left(\sum_{j=i+1}^{N-1}\hspace{-1.0mm}\left[\sum_{k=0}^{j}(3C_{\nabla f}^2)^{k}\hspace{-0.5mm}\right]\hspace{-0.5mm}(4C_{\nabla f}^2)^{j-i-1}\hspace{-1.5mm}\right)\mathbb{E}\hspace{-0.5mm}\left[  \left\| \tilde{y}_{i}^{(t)}-y_{i}^{(t)}  \right\|^4\hspace{-0.5mm} \right]\hspace{-1.0mm}.
\end{aligned}
\end{equation}
Define
\begin{align*}
     C_{y_i}^2:=&8C_{\nabla^2 f}^2L_{\nabla f}^2C_{v}^2\sum_{j=i+1}^{N-1}\left[\sum_{k=0}^{j}(3C_{\nabla f}^2)^{k}\right](4C_{\nabla f}^2)^{j-i-1}+2L_{\nabla^2 f}^2C_{v}^2\sum_{j=0}^{i}(3C_{\nabla f}^2)^{j}.
\end{align*}
Substituting~\eqref{estimation_y_1_1} into~\eqref{estimation_u_4_1} then gives
\allowdisplaybreaks
\begin{align*}
&\mathbb{E}\left[\left\|\mathbb{E}_t[\tilde{\boldsymbol{u}}^{(t)}-\boldsymbol{u}^{(t)}]\right\|^2\right]\\
\leq &8C_{\nabla^2 f}^2C_{v}^2\sum_{i=1}^{N-1}\left(\sum_{j=i}^{N-1}\left[\sum_{k=0}^{j}(3C_{\nabla f}^2)^{k}\right](4C_{\nabla f}^2)^{j-i}\right)\mathbb{E}\left[  \left\| \mathbb{E}^{\mathcal{G}_{i}}_t[\delta_{i}^{(t)}]  \right\|^2 \right]\\
&+2\hspace{-1.0mm}\sum_{i=1}^{N-1}\hspace{-0.5mm}\left[\hspace{-0.2mm}\sum_{j=1}^{i}(3C_{\nabla f}^2)^{j}\hspace{-0.5mm}\right]\hspace{-0.5mm}\mathbb{E}\hspace{-0.5mm}\left[  \left\|\mathbb{E}_t^{\mathcal{H}_{i+1}}[\varepsilon_{i+1}^{(t)}]   \right\|^2 \right]\hspace{-0.5mm}+\hspace{-0.5mm}8\hspace{-1.0mm}\sum_{i=1}^{N-1}\left[\sum_{j=i}^{N-1}C_{y_j}^2(8L_f^4)^{j-i}\right]\hspace{-0.5mm}\mathbb{E}\hspace{-0.5mm}\left[  \left\| \delta_{i}^{(t)}  \right\|^4 \right]\hspace{-1.0mm}.    
\end{align*}
Defining
\begin{subequations}
\label{eqwna}
\begin{align}
C^{b}_{\delta_i}:=&8L_{\nabla^2 f}^2C_{v}^2\sum_{j=i}^{N-1}\left[\sum_{k=0}^{j}(3C_{\nabla f}^2)^{k}\right](4C_{\nabla f}^2)^{j-i},\label{eqwna-1}\\
C^{b}_{\varepsilon_{i+1}}:=&2\sum_{j=1}^{i}(3C_{\nabla f}^2)^{j},\quad\tilde{C}^{b}_{\delta_i}:=8\sum_{j=i}^{N-1}C_{y_j}^2(8L_f^4)^{j-i},\label{eqwna-2}
\end{align}
\end{subequations}
we conclude that~\eqref{error_bias} holds for all $t = 1, 2, \ldots, T$.
\end{proof}
\end{theorem}

Theorem~\ref{thm-utilde-u} shows that the bias of the gradient estimate is controlled by the expected forward perturbations $\mathbb{E}[\delta]$, the expected backward perturbations $\mathbb{E}[\varepsilon]$, and the \emph{fourth moment} of the forward perturbations $\|\delta\|^4$. As with~\eqref{error_variance}, the coefficients in~\eqref{eqwna} exhibit exponential growth with the depth of the computational graph.

Theorems~\ref{thm-utilde-u-0} and~\ref{thm-utilde-u} together provide a complete characterization of how forward and backward perturbations corrupt gradient estimates, bounding the variance and bias of the gradient error, respectively. Combined, they yield a comprehensive quantification of the deviation between the perturbed gradient $\tilde{\boldsymbol{u}}^{(t)}$ and the ideal $\boldsymbol{u}^{(t)}$.

\section{Convergence of SGD with Perturbed Forward-Backward Passes}
\label{section: Convergence rate with computation error}
Building on the error propagation analysis in Section~\ref{section: Error propagation analysis}, we now establish convergence rates for SGD with perturbed forward and backward passes.
\subsection{Convergence with Non-Convex Assumption}
We begin with a descent lemma for SGD under nonconvex objectives and perturbed forward-backward passes.
\begin{theorem}
    Suppose Assumptions \ref{assumption:smoothness}-\ref{assumption:smoothness of the components} are all satisfied. If the step-size $\gamma\leq\frac{1}{3L_{\nabla\ell}}$, then it holds that:
    \begin{equation}
    \label{eq:l lipschitz1}
    \begin{aligned}
        \dfrac{1}{T}\sum_{t=0}^{T-1}\mathbb{E}\left[\left\Vert\nabla\ell(\w^{(t)})\right\Vert^2\right]
        \leq&\dfrac{6\Delta_0}{\gamma T}+6L_{\nabla \ell}\gamma\sigma^2+\dfrac{6L_{\nabla \ell}\gamma}{T}\sum_{t=0}^{T-1}\mathbb{E}\left[\left\Vert\tilde{\u}^{(t)}-{\u}^{(t)}\right\Vert^2\right]\\
        &+\dfrac{3}{T}\sum_{t=0}^{T-1}\mathbb{E}\left[\left\Vert\mathbb{E}_t[\tilde{\u}^{(t)}-{\u}^{(t)}]\right\Vert^2\right].
    \end{aligned}
    \end{equation}
\begin{proof}
    Since $\nabla\ell$ is $L_{\nabla \ell}$-Lipschitz continuous, we have
    \begin{equation}
    \label{eq:l lipschitz}
    \begin{aligned}
        \ell(\w^{(t+1)})\leq&\ell(\w^{(t)})+\langle \nabla\ell(\w^{(t)}),\w^{(t+1)}-\w^{(t)} \rangle+\dfrac{L_{\nabla \ell}}{2}\left\Vert\w^{(t+1)}-\w^{(t)}\right\Vert^2\\
        =&\ell(\w^{(t)})-\gamma\langle \nabla\ell(\w^{(t)}),\tilde{\u}^{(t)}+{\r}^{(t)} \rangle+\dfrac{L_{\nabla \ell}\gamma^2}{2}\left\Vert\tilde{\u}^{(t)}+{\r}^{(t)}\right\Vert^2\\
        \leq&\ell(\w^{(t)})-\gamma\langle \nabla\ell(\w^{(t)}),\u^{(t)} +{\r}^{(t)}\rangle+L_{\nabla \ell}\gamma^2\left\Vert{\u}^{(t)}+{\r}^{(t)}\right\Vert^2\\
        &+L_{\nabla \ell}\gamma^2\left\Vert\tilde{\u}^{(t)}-{\u}^{(t)}\right\Vert^2
        -\gamma\langle \nabla\ell(\w^{(t)}),\tilde{\u}^{(t)}-{\u}^{(t)} \rangle,
    \end{aligned}
    \end{equation}
    where the last inequality follows from Young's inequality. Taking the conditional expectation with respect to $\mathcal{F}^{(t)}$ yields
    \begin{equation}
    \label{eq:l lipschitz2}
    \begin{aligned}
        &\mathbb{E}_t[\ell(\w^{(t+1)})]\\
        \leq&\ell(\w^{(t)})\hspace{-0.5mm}-\hspace{-0.5mm}\gamma(1\hspace{-0.5mm}-\hspace{-0.5mm}L_{\nabla \ell}\gamma)\left\Vert\nabla\ell(\w^{(t)})\right\Vert^2\hspace{-1.5mm}+\hspace{-0.5mm}L_{\nabla \ell}\gamma^2\mathbb{E}_t\hspace{-0.5mm}\left[\hspace{-0.5mm}\left\Vert\nabla\ell(\w^{(t)})\hspace{-0.5mm}-\hspace{-0.5mm}({\u}^{(t)}\hspace{-0.5mm}+\hspace{-0.5mm}{\r}^{(t)})\right\Vert^2\right]\\
        &+L_{\nabla \ell}\gamma^2\mathbb{E}_t\left[\left\Vert\tilde{\u}^{(t)}-{\u}^{(t)}\right\Vert^2\right]-\gamma\langle \nabla\ell(\w^{(t)}),\mathbb{E}_t[\tilde{\u}^{(t)}-{\u}^{(t)} ]\rangle\\
        \leq&\ell(\w^{(t)})\hspace{-0.5mm}-\hspace{-0.5mm}\gamma(\hspace{-0.5mm}1\hspace{-0.5mm}-\hspace{-0.5mm}L_{\nabla \ell}\gamma\hspace{-0.5mm}-\hspace{-0.5mm}\dfrac{1}{2}\hspace{-0.5mm})\hspace{-0.5mm}\left\Vert\nabla\ell(\w^{(t)})\right\Vert^2\hspace{-2.0mm}+\hspace{-0.5mm}L_{\nabla \ell}\gamma^2\mathbb{E}_t\hspace{-1.0mm}\left[\left\Vert\nabla\ell(\w^{(t)})\hspace{-0.5mm}-\hspace{-0.5mm}({\u}^{(t)}+{\r}^{(t)})\right\Vert^2\right]\\
        &+L_{\nabla \ell}\gamma^2\mathbb{E}_t\left[\left\Vert\tilde{\u}^{(t)}-{\u}^{(t)}\right\Vert^2\right]+\frac{\gamma}{2}\left\Vert\mathbb{E}_t[\tilde{\u}^{(t)}-{\u}^{(t)} ]\right\Vert^2\\\leq&\ell(\w^{(t)})-\dfrac{1}{6}\gamma\left\Vert\nabla\ell(\w^{(t)})\right\Vert^2+L_{\nabla \ell}\gamma^2\sigma^2+L_{\nabla \ell}\gamma^2\mathbb{E}_t\left[\left\Vert\tilde{\u}^{(t)}-{\u}^{(t)}\right\Vert^2\right]\\
        &+\dfrac{\gamma}{2}\left\Vert\mathbb{E}_t[\tilde{\u}^{(t)}-{\u}^{(t)}]\right\Vert^2,
    \end{aligned}
    \end{equation}
    where the first inequality uses the fact that $\w^{(t)}$ is measurable with respect to $\mathcal{G}^{(t)}$ and that $\u^{(t)}$ is an unbiased estimator of $\nabla\ell(\w^{(t)})$ (Assumption~\ref{assumption:unbiased}). The last inequality follows from the bounded variance of $\u^{(t)}$ (Assumption~\ref{assumption:unbiased}) and the step size condition $L_{\nabla \ell}\gamma \leq \frac{1}{3}$. Taking the full expectation and summing over $t = 0, 1, \ldots, T-1$, we obtain
    \begin{equation*}
    \begin{aligned}
        \dfrac{1}{T}\sum_{t=0}^{T-1}\mathbb{E}\left[\left\Vert\nabla\ell(\w^{(t)})\right\Vert^2\right]
        \leq&\dfrac{6\Delta_0}{\gamma T}+6L_{\nabla \ell}\gamma\sigma^2+\dfrac{6L_{\nabla \ell}\gamma}{T}\sum_{t=0}^{T-1}\mathbb{E}\left[\left\Vert\tilde{\u}^{(t)}-{\u}^{(t)}\right\Vert^2\right]\\
        &+\dfrac{3}{T}\sum_{t=0}^{T-1}\mathbb{E}\left[\left\Vert\mathbb{E}_t[\tilde{\u}^{(t)}-{\u}^{(t)}]\right\Vert^2\right],
    \end{aligned}
    \end{equation*}
    which establishes~\eqref{eq:l lipschitz1}.
\end{proof}
\end{theorem}

Inequality~\eqref{eq:l lipschitz1} shows that the impact of forward and backward perturbations on convergence can be decomposed into two terms. The first is the variance of the gradient error, $\mathbb{E}_t[\|\tilde{\u}^{(t)} - \u^{(t)}\|^2]$, which measures the $\ell_2$ deviation between the perturbed gradient $\tilde{\u}^{(t)}$ and the ideal stochastic gradient $\u^{(t)}$. The second is the squared bias, $\|\mathbb{E}_t[\tilde{\u}^{(t)} - \u^{(t)}]\|^2$, which captures the systematic error introduced by the perturbations. Both terms were analyzed in Section~\ref{section: Error propagation analysis}. 

Substituting the bounds from~\eqref{error_variance} and~\eqref{error_bias} into~\eqref{eq:l lipschitz1} yields the following convergence rate in the nonconvex setting.
\begin{theorem}
\label{thm: 5.2}
    Suppose Assumptions \ref{assumption:smoothness}-\ref{assumption:smoothness of the components} are all satisfied. If the step-size $\gamma\leq\frac{1}{3L_{\nabla\ell}}$, then it holds that:
    \begin{equation}
    \label{eq:l lipschitz1_new}
    \begin{aligned}
        &\dfrac{1}{T}\sum_{t=0}^{T-1}\mathbb{E}\left[\left\Vert\nabla\ell(\w^{(t)})\right\Vert^2\right]\\
        \leq&\dfrac{6\Delta_0}{\gamma T}+6L_{\nabla \ell}\gamma\sigma^2+\dfrac{6L_{\nabla \ell}\gamma}{T}\sum_{t=0}^{T-1}\left(\sum_{i=1}^{N-1}C^{e}_{\delta_i}\mathbb{E}\left[  \left\| \delta_{i}^{(t)}  \right\|^2 \right]+\sum_{i=2}^{N}C^{e}_{\varepsilon_{i}}\mathbb{E}\left[  \left\| \varepsilon_{i}^{(t)}  \right\|^2 \right]\right)\\
        &\hspace{-1.6mm}+\hspace{-1.0mm}\dfrac{3}{T}\hspace{-0.5mm}\sum_{t=0}^{T-1}\hspace{-1.0mm}\left(\hspace{-0.5mm}\sum_{i=1}^{N-1}\hspace{-1.0mm}C_{\delta_i}^b\mathbb{E}\hspace{-0.5mm}\left[\hspace{-0.3mm}  \left\| \mathbb{E}^{\mathcal{G}_{i}}_t[\delta_{i}^{(t)}]  \right\|^2 \right]\hspace{-1.0mm}+\hspace{-0.5mm}\sum_{i=2}^{N}\hspace{-0.5mm}C^{b}_{\varepsilon_{i}}\mathbb{E}\hspace{-0.5mm}\left[\hspace{-0.3mm}  \left\|\mathbb{E}_t^{\mathcal{H}_{i}}[\varepsilon_{i}^{(t)}]   \right\|^2 \right]\hspace{-1.0mm}+\hspace{-1.0mm}\sum_{i=1}^{N-1}\hspace{-0.8mm}\tilde{C}_{\delta_i}^b\mathbb{E}\hspace{-0.5mm}\left[ \hspace{-0.3mm} \left\| \delta_{i}^{(t)} \right\|^4 \right]\hspace{-1.0mm}\right)\hspace{-1.0mm}.
    \end{aligned}
    \end{equation}
\end{theorem}

\begin{remark}
Setting $\delta_i^{(t)} = 0$ and $\varepsilon_i^{(t)} = 0$ recovers the standard SGD without perturbations in forward and backward passes. In this case, choosing the step size $\gamma = \mathcal{O}(T^{-\frac{1}{2}})$ in~\eqref{eq:l lipschitz1_new} yields the classical $\mathcal{O}(T^{-\frac{1}{2}})$ convergence rate of SGD in the nonconvex setting.
\end{remark}

\subsection{Convergence with PL Condition}
Analogous to the nonconvex analysis, we derive the convergence rate under the PL condition.
\begin{theorem}
\label{thm54}
    Suppose Assumption \ref{assumption:smoothness}-\ref{assumption:pl} are all satisfied. If the step-size $\gamma\leq\frac{1}{3L_{\nabla\ell}}$, then it holds that:
    \begin{equation}
    \begin{aligned}\label{eq:l lipschitz4_new}
        &\hspace{-8mm}\mathbb{E}\left[\ell(\w^{(T)})-\ell^*\right] \\
        &\leq\left(1-\dfrac{\mu\gamma}{3}\right)^T\Delta_0+\dfrac{3L_{\nabla\ell}\gamma\sigma^2}{\mu} \\
         &+L_{\nabla \ell}\gamma^2\sum_{t=0}^{T-1}\left(1-\dfrac{\mu\gamma}{3}\right)^{T-t}\left(\sum_{i=1}^{N-1}C^{e}_{\delta_i}\mathbb{E}\left[  \left\| \delta_{i}^{(t)}  \right\|^2 \right]+\sum_{i=2}^{N}C^{e}_{\varepsilon_{i}}\mathbb{E}\left[  \left\| \varepsilon_{i}^{(t)}  \right\|^2 \right]\right)\\
        &+\dfrac{\gamma}{2}\sum_{t=0}^{T-1}\left(1-\dfrac{\mu\gamma}{3}\right)^{T-t}\left(\sum_{i=1}^{N-1}C_{\delta_i}^b\mathbb{E}\left[  \left\| \mathbb{E}^{\mathcal{G}_{i}}_t[\delta_{i}^{(t)}]  \right\|^2 \right]+\sum_{i=2}^{N}C^{b}_{\varepsilon_{i}}\mathbb{E}\left[  \left\|\mathbb{E}_t^{\mathcal{H}_{i}}[\varepsilon_{i}^{(t)}]   \right\|^2 \right]\right.\\
        &\quad\left.+\sum_{i=1}^{N-1}\tilde{C}_{\delta_i}^b\mathbb{E}\left[  \left\| \delta_{i}^{(t)}  \right\|^4 \right]\right).
    \end{aligned}
    \end{equation}
\begin{proof}
    Under Assumption~\ref{assumption:unbiased}, it follows from~\eqref{eq:l lipschitz2} that
    {
    \begin{equation}
    \label{eq:l lipschitz21}
    \begin{aligned}
        \mathbb{E}_t\left[\ell(\w^{(t+1)})\right]
        \leq&\ell(\w^{(t)})-\dfrac{\mu\gamma}{3}(\ell(\w^{(t)})-\ell^*)+L_{\nabla\ell}\gamma^2\sigma^2\\
        &+L_{\nabla \ell}\gamma^2\mathbb{E}_t\left[\left\Vert\tilde{\u}^{(t)}-{\u}^{(t)}\right\Vert^2\right]+\dfrac{\gamma}{2}\mathbb{E}_t\left\Vert\mathbb{E}_t[\tilde{\u}^{(t)}-{\u}^{(t)}]\right\Vert^2.
    \end{aligned}
    \end{equation}}
Taking the full expectation on both sides yields
    \begin{equation}
    \label{eq:l lipschitz3}
    \begin{aligned}
        \mathbb{E}\left[\ell(\w^{(t+1)})-\ell^*\right]
        \leq&\left(1-\dfrac{\mu\gamma}{3}\right)\mathbb{E}\left[\ell(\w^{(t)})-\ell^*\right]+L_{\nabla\ell}\gamma^2\sigma^2\\
        &+L_{\nabla \ell}\gamma^2\mathbb{E}\left[\left\Vert\tilde{\u}^{(t)}-{\u}^{(t)}\right\Vert^2\right]+\dfrac{\gamma}{2}\mathbb{E}\left[\left\Vert\mathbb{E}_t[\tilde{\u}^{(t)}-{\u}^{(t)}]\right\Vert^2\right].
    \end{aligned}
    \end{equation}
Applying this recursively, we obtain
    \begin{equation}
        \label{eq:new}
    \begin{aligned}
        &\mathbb{E}\left[\ell(\w^{(T)})-\ell^*\right]\\
        \leq&\left(1-\dfrac{\mu\gamma}{3}\right)^T\Delta_0+\dfrac{3L_{\nabla\ell}\gamma\sigma^2}{\mu}+L_{\nabla \ell}\gamma^2\sum_{t=0}^{T-1}\left(1-\dfrac{\mu\gamma}{3}\right)^{T-t}\mathbb{E}\left[\left\Vert\tilde{\u}^{(t)}-{\u}^{(t)}\right\Vert^2\right]\\
        &+\dfrac{\gamma}{2}\sum_{t=0}^{T-1}\left(1-\dfrac{\mu\gamma}{3}\right)^{T-t}\mathbb{E}\left[\left\Vert\mathbb{E}_t[\tilde{\u}^{(t)}-{\u}^{(t)}]\right\Vert^2\right].
    \end{aligned}
    \end{equation}
Finally, substituting~\eqref{error_variance} and~\eqref{error_bias} into~\eqref{eq:new} gives
    \begin{equation}
    \label{eq:l lipschitz4_new2}
    \begin{aligned}
        \mathbb{E}\left[\ell(\w^{(T)})-\ell^*\right]
        \leq&\left(1-\dfrac{\mu\gamma}{3}\right)^T\Delta_0+\dfrac{3L_{\nabla\ell}\gamma\sigma^2}{\mu}\\
        +&L_{\nabla \ell}\gamma^2\sum_{t=0}^{T-1}\left(1-\dfrac{\mu\gamma}{3}\right)^{T-t}\left(\sum_{i=1}^{N-1}C^{e}_{\delta_i}\mathbb{E}\left[  \left\| \delta_{i}^{(t)}  \right\|^2 \right]+\sum_{i=2}^{N}C^{e}_{\varepsilon_{i}}\mathbb{E}\left[  \left\| \varepsilon_{i}^{(t)}  \right\|^2 \right]\right)\\
        +&\dfrac{\gamma}{2}\sum_{t=0}^{T-1}\left(1-\dfrac{\mu\gamma}{3}\right)^{T-t}\left(\sum_{i=1}^{N-1}C_{\delta_i}^b\mathbb{E}\left[  \left\| \mathbb{E}^{\mathcal{G}_{i}}_t[\delta_{i}^{(t)}]  \right\|^2 \right]+\sum_{i=2}^{N}C^{b}_{\varepsilon_{i}}\mathbb{E}\left[  \left\|\mathbb{E}_t^{\mathcal{H}_{i}}[\varepsilon_{i}^{(t)}]   \right\|^2 \right]\right.\\
        &\quad\left.+\sum_{i=1}^{N-1}\tilde{C}_{\delta_i}^b\mathbb{E}\left[  \left\| \delta_{i}^{(t)}  \right\|^4 \right]\right),
    \end{aligned}
    \end{equation}
which establishes~\eqref{eq:l lipschitz4_new}.
\end{proof}
\end{theorem}

\begin{remark}
Setting $\delta_i^{(t)} = 0$ and $\varepsilon_i^{(t)} = 0$ recovers the standard SGD without perturbations in forward and backward passes. In this case, choosing the step size $\gamma = \mathcal{O}(T^{-1}\ln T)$ in~\eqref{eq:l lipschitz4_new} yields the classical $\tilde{\mathcal{O}}(T^{-1})$ convergence rate of SGD under the PL condition.
\end{remark}

\subsection{Interpreting the Convergence Theorems}
Theorems~\ref{thm: 5.2} and~\ref{thm54} offer a clean separation among optimization progress, stochastic sampling noise, and endogenous perturbation effects induced by corrupted forward/backward passes, providing a structured framework for examining the distinct role of each component.

\vspace{1mm}
\noindent \textbf{Exponential coefficient scaling with depth.}
The coefficients presented in Theorems~\ref{thm: 5.2} and Theorems~\ref{thm54}, namely $C^{e}_{\delta_i},C^{e}_{\varepsilon_i},C^{b}_{\delta_i},C^{b}_{\varepsilon_i},\widetilde C^{b}_{\delta_i}$, encode how perturbations at individual operators impact the gradient error. As the explicit formulas in Theorems~\ref{thm-utilde-u-0} and~\ref{thm-utilde-u} reveal, these coefficients take a geometric-series form and can therefore scale \emph{exponentially} in the depth~$N$ whenever per-operator Lipschitz factors exceed~$1$ (see Section~\ref{section: Error propagation analysis}).

\vspace{1mm}
\noindent \textbf{Structural asymmetry between forward and backward perturbations.}
No single scalar ordering holds across all regimes, as the coefficients are operator- and depth-dependent. Nevertheless, Theorems~\ref{thm: 5.2} and~\ref{thm54} imply a robust qualitative hierarchy: \emph{forward-pass perturbations~$\delta$ can introduce non-vanishing bias through higher-order effects even when zero-mean, whereas backward-pass perturbations~$\varepsilon$ do not.} This asymmetry is the primary structural reason why Section~\ref{section: Recover from computation error and gradient spike} permits much weaker constraints on~$\delta$ than on~$\varepsilon$ (see further discussion in Sections~\ref{subsection:continous error} and~\ref{subsection:intermittent error}).

\vspace{1mm}
\noindent \textbf{Sensitivity to perturbation location.}
Forward perturbations near the input (i.e., operators with small index~$i$) propagate through many subsequent operators, whereas backward perturbations near the output (i.e., operators with large index~$i$) propagate through many upstream Jacobian factors. 
Consequently, when only a single layer is perturbed, the two cases exhibit opposite depth hierarchies: for \emph{forward} perturbations~$(\delta_1,\delta_{N-1})$, the shallow error~$\delta_1$ is typically far more harmful than the deep error~$\delta_{N-1}$; conversely, for \emph{backward} perturbations~$(\varepsilon_N,\varepsilon_2)$, the deep error~$\varepsilon_N$ is typically far more harmful than its shallow counterpart~$\varepsilon_2$. This hierarchy is  captured by the layer-dependent coefficients~$C^{e}_{\delta_i}$ and~$C^{e}_{\varepsilon_{i+1}}$ in Theorems~\ref{thm: 5.2} and~\ref{thm54}.

\section{Conditions Ensuring Convergence under Perturbations}
\label{section: Recover from computation error and gradient spike}
This section identifies conditions on the magnitude and frequency of perturbations under which SGD still converges. We focus on two types of perturbations:

\begin{itemize}
\vspace{1mm}
\item[\textbf{(1)}] \textbf{Frequent perturbations:} Perturbations occur at every iteration with unknown energy, as in communication errors in pipeline compression~\cite{wang2022fine}. We seek to characterize the energy levels that still guarantee convergence.

\item[\textbf{(2)}] \textbf{Intermittent perturbation: } Perturbations occur sporadically with fixed energy, as in bit flips. We seek to characterize the frequency of occurrence that still guarantees convergence. 
\end{itemize}
\subsection{Convergence Conditions for Frequent Perturbation}
\label{subsection:continous error}
The following corollary provides sufficient conditions on the energy of frequent perturbations that guarantee convergence.

\begin{corollary}
\label{corollary:continous error}
    Under Assumptions~\ref{assumption:smoothness}--\ref{assumption:smoothness of the components}, SGD achieves a convergence rate of $\mathcal{O}(T^{-\frac{1}{2}})$ if for all $t = 1, 2, \ldots, T$, the perturbations $\delta_i^{(t)}$ and $\varepsilon_i^{(t)}$ satisfy
    \begin{align}
    \label{continus_nonconvex}
        \left\| \delta_{i}^{(t)}  \right\|\lesssim T^{-\frac{1}{8}},\quad\left\| \mathbb{E}^{\mathcal{G}_{i}}_t[\delta_{i}^{(t)}]  \right\| \lesssim T^{-\frac{1}{4}},\quad \left\| \varepsilon_{i}^{(t)}  \right\|\lesssim 1,\quad \left\| \mathbb{E}^{\mathcal{H}_{i}}_t[\varepsilon_{i}^{(t)}]  \right\| \lesssim T^{-\frac{1}{4}}.
    \end{align}
    Moreover, if Assumption~\ref{assumption:pl} also holds, then SGD achieves a convergence rate of $\tilde{\mathcal{O}}(T^{-1})$ if for all $t = 1, 2, \ldots, T$, the perturbations satisfy
    \begin{equation}
    \begin{aligned}
    \label{continus_pl}
        \left\| \delta_{i}^{(t)}  \right\|&\lesssim T^{-\frac{1}{4}}\ln^{\frac{1}{4}}T,\quad \left\| \mathbb{E}^{\mathcal{G}_{i}}_t[\delta_{i}^{(t)}]  \right\| \lesssim T^{-\frac{1}{2}}\ln^{\frac{1}{2}}T,\\
        \left\| \varepsilon_{i}^{(t)}  \right\|&\lesssim 1,\hspace{17.5mm} \left\| \mathbb{E}^{\mathcal{H}_{i}}_t[\varepsilon_{i}^{(t)}]  \right\| \lesssim T^{-\frac{1}{2}}\ln^{\frac{1}{2}}T.
    \end{aligned}
    \end{equation}
\begin{proof}
    Setting $\gamma = \mathcal{O}(T^{-\frac{1}{2}})$ and substituting~\eqref{continus_nonconvex} into~\eqref{eq:l lipschitz1_new} yields
    \begin{equation}
    \label{eq:6.3}
    \begin{aligned}
        &\dfrac{1}{T}\sum_{t=0}^{T-1}\mathbb{E}\left[\left\Vert\nabla\ell(\w^{(t)})\right\Vert^2\right]\\
        \lesssim&\dfrac{1}{\sqrt{T}}+\dfrac{1}{T^{\frac{3}{2}}}
        \sum_{t=0}^{T-1}\left(\sum_{i=1}^{N-1}\mathbb{E}\left[  \left\| \delta_{i}^{(t)}  \right\|^2 \right]+\sum_{i=2}^{N}\mathbb{E}\left[  \left\| \varepsilon_{i}^{(t)}  \right\|^2 \right]\right)\\
        &+\dfrac{1}{T}\sum_{t=0}^{T-1}\left(\sum_{i=1}^{N-1}\mathbb{E}\left[  \left\| \mathbb{E}^{\mathcal{G}_{i}}_t[\delta_{i}^{(t)}]  \right\|^2 \right]+\sum_{i=2}^{N}\mathbb{E}\left[  \left\|\mathbb{E}_t^{\mathcal{H}_{i}}[\varepsilon_{i}^{(t)}]   \right\|^2 \right]+\sum_{i=1}^{N-1}\mathbb{E}\left[  \left\| \delta_{i}^{(t)}  \right\|^4 \right]\right).
    \end{aligned}
    \end{equation}
    To achieve the $\mathcal{O}(T^{-\frac{1}{2}})$ rate, it suffices to have
    \begin{subequations}
    \label{condition_nonconvex_frequent}
    \begin{align}
        &\sum_{t=0}^{T-1}\left(\sum_{i=1}^{N-1}\mathbb{E}\left[  \left\| \delta_{i}^{(t)}  \right\|^2 \right]+\sum_{i=2}^{N}\mathbb{E}\left[  \left\| \varepsilon_{i}^{(t)}  \right\|^2 \right]\right)\lesssim T,\\
        &\sum_{t=0}^{T-1}\left(\sum_{i=1}^{N-1}\mathbb{E}\left[  \left\| \mathbb{E}^{\mathcal{G}_{i}}_t[\delta_{i}^{(t)}]  \right\|^2 \right]+\sum_{i=2}^{N}\mathbb{E}\left[  \left\|\mathbb{E}_t^{\mathcal{H}_{i}}[\varepsilon_{i}^{(t)}]   \right\|^2 \right]+\sum_{i=1}^{N-1}\mathbb{E}\left[  \left\| \delta_{i}^{(t)}  \right\|^4 \right]\right)\lesssim T^{\frac{1}{2}}.
    \end{align}        
    \end{subequations}
    If $\delta_i^{(t)}$ and $\varepsilon_i^{(t)}$ satisfy~\eqref{continus_nonconvex} for $t = 1, 2, \ldots, T$, then~\eqref{condition_nonconvex_frequent} holds, yielding $\mathcal{O}(T^{-\frac{1}{2}})$ rate.
    
    Now suppose Assumption~\ref{assumption:pl} also holds. Setting $\gamma = \mathcal{O}(T^{-1}\ln T)$ and substituting~\eqref{continus_pl} into~\eqref{eq:l lipschitz4_new} gives
    \begin{equation}
    \begin{aligned}
        &\mathbb{E}\left[\ell(\w^{(T)})-\ell^*\right]
        \lesssim\dfrac{1}{T}+\dfrac{\ln T}{T}\\
        +&\dfrac{\ln^2T}{T^2}\sum_{t=0}^{T-1}\left(1-\dfrac{\ln T}{T}\right)^{T-t}\hspace{-1mm}\left(\sum_{i=1}^{N-1}\mathbb{E}\left[  \left\| \delta_{i}^{(t)}  \right\|^2 \right]\hspace{-1mm}+\hspace{-1mm}\sum_{i=2}^{N}\mathbb{E}\left[  \left\| \varepsilon_{i}^{(t)}  \right\|^2 \right]\right)\\
        +&\dfrac{\ln T}{T}\hspace{-0.5mm}\sum_{t=0}^{T-1}\hspace{-0.5mm}\left(\hspace{-1.0mm}1\hspace{-0.5mm}-\hspace{-0.5mm}\dfrac{\ln T}{T}\hspace{-0.5mm}\right)^{\hspace{-1.0mm}T-t}\hspace{-1.5mm}\left(\sum_{i=1}^{N-1}\hspace{-0.5mm}\mathbb{E}\hspace{-0.5mm}\left[ \hspace{-0.5mm} \left\| \mathbb{E}^{\mathcal{G}_{i}}_t[\delta_{i}^{(t)}]  \right\|^2 \right]\hspace{-1mm}+\hspace{-1mm}\sum_{i=2}^{N}\hspace{-0.5mm}\mathbb{E}\hspace{-0.5mm}\left[\hspace{-0.5mm}  \left\|\mathbb{E}_t^{\mathcal{H}_{i}}[\varepsilon_{i}^{(t)}]   \right\|^2 \right]\hspace{-1mm}+\hspace{-1mm}\sum_{i=1}^{N-1}\hspace{-0.5mm}\mathbb{E}\hspace{-0.5mm}\left[ \hspace{-0.5mm} \left\| \delta_{i}^{(t)}  \right\|^4 \right]\hspace{-1.0mm}\right)\hspace{-1.0mm}.
    \end{aligned}
    \end{equation}
    To achieve the $\tilde{\mathcal{O}}(T^{-1})$ rate, it suffices to have
    \begin{subequations}
\label{condition_pl_frequent}
    \begin{align}
        &\sum_{i=1}^{N-1}\mathbb{E}\left[  \left\| \delta_{i}^{(t)}  \right\|^2 \right]+\sum_{i=2}^{N}\mathbb{E}\left[  \left\| \varepsilon_{i}^{(t)}  \right\|^2 \right]\lesssim 1,\\
        &\sum_{i=1}^{N-1}\mathbb{E}\left[  \left\| \mathbb{E}^{\mathcal{G}_{i}}_t[\delta_{i}^{(t)}]  \right\|^2 \right]+\sum_{i=2}^{N}\mathbb{E}\left[  \left\|\mathbb{E}_t^{\mathcal{H}_{i}}[\varepsilon_{i}^{(t)}]   \right\|^2 \right]+\sum_{i=1}^{N-1}\mathbb{E}\left[  \left\| \delta_{i}^{(t)}  \right\|^4 \right]\lesssim \dfrac{\ln T}{T},
    \end{align}        
    \end{subequations}
    where we have used the fact that
    \begin{align*}
        \sum_{t=0}^{T-1}\left(1-\dfrac{\ln T}{T}\right)^{T-t}\lesssim\dfrac{\ln T}{T}.
    \end{align*}
    If $\delta_i^{(t)}$ and $\varepsilon_i^{(t)}$ satisfy~\eqref{continus_pl} for $t = 1, 2, \ldots, T$, then~\eqref{condition_pl_frequent} holds, yielding $\tilde{\mathcal{O}}(T^{-1})$ rate.
\end{proof}
\end{corollary}

\begin{remark}
\label{rem:freq-thresholds}
Corollary~\ref{corollary:continous error} quantifies admissible magnitudes of \emph{frequent} perturbations along the forward and backward passes.
From~\eqref{continus_nonconvex} and~\eqref{continus_pl}, one sees that the conditional bias terms
$\|\mathbb{E}^{\mathcal{G}_{i}}_t[\delta_{i}^{(t)}]\|$ and $\|\mathbb{E}^{\mathcal{H}_{i}}_t[\varepsilon_{i}^{(t)}]\|$
are subject to stricter requirements than the corresponding second-moment terms
$\|\delta_{i}^{(t)}\|$ and $\|\varepsilon_{i}^{(t)}\|$.
More importantly, the forward perturbations $\delta_i^{(t)}$ must satisfy stronger conditions than the backward perturbations $\varepsilon_i^{(t)}$
to retain the error-free rate.
For example, when both perturbations are unbiased (i.e., $\mathbb{E}^{\mathcal{G}_{i}}_t[\delta_{i}^{(t)}]=0$ and $\mathbb{E}^{\mathcal{H}_{i}}_t[\varepsilon_{i}^{(t)}]=0$), SGD on nonconvex objectives can still achieve $\mathcal{O}(T^{-1/2})$ with $\|\varepsilon_i^{(t)}\|=\mathcal{O}(1)$,
but generally fails to do so with $\|\delta_i^{(t)}\|=\mathcal{O}(1)$.
This forward/backward asymmetry will be further reflected in the remark below.
\end{remark}

\begin{remark}\label{rmk-requrent}
Inequality~\eqref{eq:6.3} suggests decomposing the   bound into two components: an optimization error that vanishes with increased optimization effort (smaller $\gamma$ and/or larger $T$), and a perturbation-induced perturbation error. As noted in Remark~\ref{rem:freq-thresholds}, under frequent perturbations, forward errors constitute the primary obstruction to eliminating this perturbation error: unless the forward perturbations satisfy the stricter decay conditions, the bound predicts a non-vanishing plateau even as $\gamma \to 0$ or $T \to \infty$. By contrast, in the unbiased regime, backward perturbations are comparatively benign and contribute primarily to the optimization error. In Section~\ref{section: Experiments}, we illustrate this distinction explicit through two parameter sweeps: a coupled sweep with $\gamma = 1/\sqrt{T}$, and a fixed-horizon sweep with common $T$ and varying $\gamma$.
\end{remark}

\subsection{Convergence Conditions for Intermittent Perturbation}
\label{subsection:intermittent error}
In this subsection, we consider the convergence of SGD under intermittent perturbations. For each time step $t$ and component $i \in \{1, \cdots, N\}$, we say that an intermittent perturbation occurs in the forward pass if $\delta_{i}^{(t)}$ is of order $\mathcal{O}(1)$; otherwise, $\delta_{i}^{(t)} = 0$. Similarly, for the backward pass, an intermittent perturbation occurs if $\varepsilon_{i}^{(t)}$ is of order $\mathcal{O}(1)$; otherwise, $\varepsilon_{i}^{(t)} = 0$. We let $Q_{\delta}$ and $Q_{\varepsilon}$ denote the total number of perturbation occurrences throughout optimization in the forward and backward passes, respectively.

The following two corollaries establish upper bounds on the number of admissible intermittent perturbation occurrences that preserve the convergence rate, under the assumption that the perturbations are zero-mean.
\begin{corollary}
\label{corollary:spike_unbiased_nonconvex}
Suppose the perturbations $\delta_i^{(t)}$ and $\varepsilon_{i+1}^{(t)}$ are zero-mean for all $t = 1, 2, \ldots, T$ and $i = 1, 2, \ldots, N-1$. Then, under Assumptions~\ref{assumption:smoothness}--\ref{assumption:smoothness of the components}, SGD achieves a convergence rate of $\mathcal{O}(T^{-\frac{1}{2}})$ provided that $Q_{\delta} \lesssim T^{\frac{1}{2}}$ and $Q_{\varepsilon} \lesssim T$.
\begin{proof}
    Substituting $\mathbb{E}^{\mathcal{G}_{i}}_t[\delta_{i}^{(t)}] = 0$, $\mathbb{E}_t^{\mathcal{H}_{i}}[\varepsilon_{i}^{(t)}] = 0$, and $\gamma = \mathcal{O}(T^{-\frac{1}{2}})$ into~\eqref{eq:l lipschitz1_new}, and using the definitions of $Q_{\delta}$ and $Q_{\varepsilon}$, we obtain
    \begin{equation}
    \begin{aligned}
        &\dfrac{1}{T}\sum_{t=0}^{T-1}\mathbb{E}\left[\left\Vert\nabla\ell(\w^{(t)})\right\Vert^2\right]\\
        \lesssim&\dfrac{1}{\sqrt{T}}\hspace{-0.5mm}+\hspace{-0.5mm}\dfrac{1}{T^{\frac{3}{2}}}\hspace{-1.0mm}\sum_{t=0}^{T-1}\hspace{-1.0mm}\left(\sum_{i=1}^{N-1}\hspace{-0.5mm}\mathbb{E}\hspace{-0.5mm}\left[  \left\| \delta_{i}^{(t)}  \right\|^2 \right]\hspace{-1.0mm}+\hspace{-1.0mm}\sum_{i=2}^{N}\hspace{-0.5mm}\mathbb{E}\hspace{-0.5mm}\left[  \left\| \varepsilon_{i}^{(t)}  \right\|^2 \right]\hspace{-0.5mm}\right)\hspace{-0.5mm}+\hspace{-0.5mm}\dfrac{1}{T}\hspace{-0.5mm}\sum_{t=0}^{T-1}\sum_{i=1}^{N-1}\hspace{-0.5mm}\mathbb{E}\hspace{-0.5mm}\left[  \left\| \delta_{i}^{(t)}  \right\|^4 \right]\\
        \lesssim&\dfrac{1}{\sqrt{T}}+\dfrac{Q_{\delta}+Q_{\varepsilon}}{T^{\frac{3}{2}}}+\dfrac{Q_{\delta}}{T}.
    \end{aligned}
    \end{equation}
    Therefore, if $Q_{\delta} \lesssim T^{\frac{1}{2}}$ and $Q_{\varepsilon} \lesssim T$, we have
    \begin{align*}
        \dfrac{1}{T}\sum_{t=0}^{T-1}\mathbb{E}\left[\left\Vert\nabla\ell(\w^{(t)})\right\Vert^2\right]\lesssim T^{-\frac{1}{2}}.
    \end{align*}
\end{proof}
\end{corollary}

\begin{corollary}
\label{corollary:spike_unbiased_pl}
    Suppose the perturbations $\delta_i^{(t)}$ and $\varepsilon_{i+1}^{(t)}$ are zero-mean for all $t = 1, 2, \ldots, T$ and $i = 1, 2, \ldots, N-1$. Then, under Assumptions~\ref{assumption:smoothness}--\ref{assumption:pl}, SGD achieves a convergence rate of $\tilde{\mathcal{O}}(T^{-1})$ provided that $Q_{\delta} \lesssim 1$ and $Q_{\varepsilon} \lesssim T$.
\begin{proof}
    Under the PL condition, substituting $\mathbb{E}^{\mathcal{G}_{i}}_t[\delta_{i}^{(t)}] = 0$, $\mathbb{E}_t^{\mathcal{H}_{i}}[\varepsilon_{i}^{(t)}] = 0$, and $\gamma = \mathcal{O}(T^{-1}\ln T)$ into~\eqref{eq:l lipschitz1_new}, and using the definitions of $Q_{\delta}$ and $Q_{\varepsilon}$, we obtain
    \begin{equation*}
    \begin{aligned}
        \mathbb{E}\hspace{-0.5mm}\left[\ell(\w^{(T)})\hspace{-0.5mm}-\hspace{-0.5mm}\ell^*\right]\hspace{-1.0mm}
        \lesssim&\dfrac{\ln T}{T}\hspace{-0.5mm}+\hspace{-0.5mm}\dfrac{\ln^2T}{T^2}\hspace{-0.5mm}\sum_{t=0}^{T-1}\hspace{-0.5mm}\left(1\hspace{-0.5mm}-\hspace{-0.5mm}\dfrac{\ln T}{T}\right)^{\hspace{-1.0mm}T-t}\hspace{-1mm}\left(\sum_{i=1}^{N-1}\mathbb{E}\hspace{-0.5mm}\left[  \left\| \delta_{i}^{(t)}  \right\|^2 \right]\hspace{-1mm}+\hspace{-1mm}\sum_{i=2}^{N}\mathbb{E}\hspace{-0.5mm}\left[  \left\| \varepsilon_{i}^{(t)}  \right\|^2 \right]\hspace{-0.5mm}\right)\\
        &+\dfrac{\ln T}{T}\sum_{t=0}^{T-1}\left(1-\dfrac{\ln T}{T}\right)^{T-t}\sum_{i=1}^{N-1}\mathbb{E}\left[  \left\| \delta_{i}^{(t)}  \right\|^4 \right]\\
        \lesssim& \dfrac{\ln T}{T}+\dfrac{\ln^2T}{T^2}\sum_{t=1}^{\max\{Q_{\delta}, Q_{\varepsilon}\}}\left(1-\dfrac{\ln T}{T}\right)^{t}+\dfrac{\ln T}{T}\sum_{t=1}^{Q_{\delta}}\left(1-\dfrac{\ln T}{T}\right)^{t}.
    \end{aligned}
    \end{equation*}
Evaluating the geometric sums yields
    \begin{equation}
    \begin{aligned}
        \mathbb{E}\left[\ell(\w^{(T)})-\ell^*\right]
        \lesssim&\dfrac{\ln T}{T}+\dfrac{\ln^2T}{T^2}\cdot\dfrac{1-\left(1-{\ln T}/{T}\right)^{\max\{Q_{\delta}, Q_{\varepsilon}\}}}{1-\left(1-{\ln T}/{T}\right)}\\
        &+\dfrac{\ln T}{T}\cdot\dfrac{1-\left(1-{\ln T}/{T}\right)^{Q_{\delta}}}{1-\left(1-{\ln T}/{T}\right)}\\
        \lesssim&\dfrac{\ln T}{T}(1+Q_{\delta})+\dfrac{\ln^2T}{T^2}\max\{Q_{\delta}+Q_{\varepsilon}\},
    \end{aligned}
    \end{equation}
    where the last inequality follows from the fact that $(1 - \ln T / T)^{\alpha} \geq 1 - \alpha \ln T / T$ for any $\alpha \geq 1$. Therefore, if $Q_{\delta} \lesssim 1$ and $Q_{\varepsilon}  \lesssim T$, we have
    \begin{align*}
        \mathbb{E}\left[\ell(\w^{(T)})-\ell^*\right]\lesssim \dfrac{\ln T}{T}.
    \end{align*}

\end{proof}
\end{corollary}

The following corollary establishes upper bounds on the number of admissible intermittent perturbation occurrences \emph{without} the zero-mean assumption, under nonconvex and PL settings respectively.
\begin{corollary}
\label{corollary:spike_biased}
    Under Assumptions~\ref{assumption:smoothness}--\ref{assumption:smoothness of the components}, SGD achieves a convergence rate of $\mathcal{O}(T^{-\frac{1}{2}})$ if $Q_{\delta} \lesssim T^{\frac{1}{2}}$ and $Q_{\varepsilon} \lesssim T^{\frac{1}{2}}$.
Moreover, if Assumption~\ref{assumption:pl} also holds, then SGD achieves a convergence rate of $\tilde{\mathcal{O}}(T^{-1})$ if $Q_{\delta} = \mathcal{O}(1)$ and $Q_{\varepsilon} = \mathcal{O}(1)$.
\begin{proof}
    Substituting $\gamma = \mathcal{O}(T^{-\frac{1}{2}})$ into~\eqref{eq:l lipschitz1_new} and using the definitions of $Q_{\delta}$ and $Q_{\varepsilon}$, we obtain
    \begin{equation}
    \begin{aligned}
        &\dfrac{1}{T}\sum_{t=0}^{T-1}\mathbb{E}\left[\left\Vert\nabla\ell(\w^{(t)})\right\Vert^2\right]
        \lesssim\dfrac{1}{\sqrt{T}}+\dfrac{Q_{\delta}+Q_{\varepsilon}}{T^{\frac{3}{2}}}+\dfrac{Q_{\delta}+Q_{\varepsilon}}{T}.
    \end{aligned}
    \end{equation}
    Therefore, if $Q_{\delta} \lesssim T^{\frac{1}{2}}$ and $Q_{\varepsilon} \lesssim T^{\frac{1}{2}}$, we have
    \begin{align*}
        \dfrac{1}{T}\sum_{t=0}^{T-1}\mathbb{E}\left[\left\Vert\nabla\ell(\w^{(t)})\right\Vert^2\right]\lesssim \dfrac{1}{\sqrt{T}}.
    \end{align*}
    
    Now suppose Assumption~\ref{assumption:pl} also holds. Substituting $\gamma = \mathcal{O}(T^{-1}\ln T)$ into~\eqref{eq:l lipschitz4_new} and using the definitions of $Q_{\delta}$ and $Q_{\varepsilon}$, we obtain
    \begin{equation}
    \begin{aligned}
        \mathbb{E}\left[\ell(\w^{(T)})-\ell^*\right]
        \lesssim&\dfrac{\ln T}{T}+\dfrac{\ln T}{T}\sum_{t=1}^{\max\{Q_{\delta}, Q_{\varepsilon}\}}\left(1-\dfrac{\ln T}{T}\right)^{t}\\
        \lesssim&\dfrac{\ln T}{T}\hspace{-1mm}\left(\hspace{-1mm}1\hspace{-0.5mm}+\hspace{-0.5mm}\dfrac{1-\left(1-{\ln T}/{T}\right)^{\max\{Q_{\delta}, Q_{\varepsilon}\}}}{1-\left(1-{\ln T}/{T}\right)}\hspace{-1mm}\right)\hspace{-1mm}
        \lesssim\dfrac{\ln T}{T}(\max\{Q_{\delta}, Q_{\varepsilon}\}\hspace{-0.5mm}+\hspace{-0.5mm}1).
    \end{aligned}
    \end{equation}
    Therefore, if $Q_{\delta} \lesssim 1$ and $Q_{\varepsilon} \lesssim 1$, we have
    \begin{align*}
        \mathbb{E}\left[\ell(\w^{(T)})-\ell^*\right]\lesssim\dfrac{\ln T}{T},
    \end{align*}
    which completes the proof.
\end{proof}
\end{corollary}

\begin{table}[t]
\centering
\label{table: error recover}
\caption{Upper bounds on the number of $\mathcal{O}(1)$-magnitude intermittent perturbation occurrences in the forward and backward passes that SGD can tolerate without deteriorating the convergence rate. ``Zero mean'' indicates that the intermittent 
perturbations have zero expectation.}
\renewcommand{\arraystretch}{1.3}
\begin{threeparttable}
\begin{tabular}{lcccc}
\toprule
Assumption                    & Convergence rate & Zero mean & $Q_{\delta}$ & $Q_{\varepsilon}$ \\ \midrule
\multirow{2}{*}{Non-convex}   & \multirow{2}{*}{$\mathcal{O}\left(\dfrac{1}{\sqrt{T}}\right)$}                 &  \cmark         & $\mathcal{O}(T^{\frac{1}{2}})$ & $\mathcal{O}(T)$ \\ \cline{3-5} 
                              &                  &  \xmark         & $\mathcal{O}(T^{\frac{1}{2}})$ & $\mathcal{O}(T^{\frac{1}{2}})$ \\ \midrule
\multirow{2}{*}{PL condition} & \multirow{2}{*}{$\mathcal{O}\left(\dfrac{\ln T}{T}\right)$}                 &      \cmark      &  $\mathcal{O}(1)$ & $\mathcal{O}(T)$  \\ \cline{3-5} 
                              &                  &  \xmark         & $\mathcal{O}(1)$ & $\mathcal{O}(1)$ \\ \bottomrule
\end{tabular}
\end{threeparttable}
\end{table}

Table~\ref{table: error recover} summarizes the admissible occurrence budgets of intermittent perturbations under different assumptions,
as derived from Corollaries~\ref{corollary:spike_unbiased_nonconvex}, \ref{corollary:spike_unbiased_pl}, and~\ref{corollary:spike_biased}.
Without the zero-mean assumption, the admissible occurrence thresholds for forward and backward perturbations coincide under both the nonconvex and PL settings.
In contrast, under zero-mean perturbations---i.e., when each perturbed forward activation and backward chain computation is (conditionally) unbiased---a pronounced asymmetry emerges:
the admissible occurrence frequency of backward perturbations can be relaxed to $\mathcal{O}(T)$, meaning that unbiased backward perturbations may persist throughout optimization without changing the convergence rate order.
For forward perturbations, however, the maximum allowable occurrence frequencies that preserve convergence are much smaller: $O(T^{1/2})$ in the nonconvex setting and $O(1)$ under the PL condition; exceeding these thresholds may lead to slower convergence or divergence (see Appendix~\ref{section: A2} for a one-dimensional example where persistent $O(1)$ forward perturbations cause convergence to a biased limit point).

\begin{remark}\label{rmk-intermittent}
In the intermittent regime, the admissible occurrence-budget bounds in Table~\ref{table: error recover}  are most naturally expressed in terms of the total perturbation counts $(Q_\delta, Q_\epsilon)$ over the horizon $T$. For instance, if a forward perturbation is injected once every $\Delta t$ steps, then $Q_\delta \approx \lceil T/\Delta t\rceil$ (and analogously for $Q_\varepsilon$), so varying $\Delta t$ directly controls the effective perturbation budget. The asymmetry highlighted in Table~\ref{table: error recover} can therefore be recast as an asymmetric budget requirement: in the unbiased setting, forward perturbations of order $\mathcal{O}(1)$ must be sufficiently infrequent (i.e., $Q_\delta$ must remain small), whereas backward perturbations may occur considerably more often (i.e., $Q_\varepsilon$ may be substantially larger). This budget-centric perspective foreshadows the phase-transition behavior observed in Section~\ref{section:7.2.3} as $\Delta t$ increases.
\end{remark}

\subsection{Convergence with Gradient Spikes}
In this subsection, we characterize conditions under which the optimization process still converges in the presence of gradient spikes. We say that the evaluated gradient at iteration $t$, denoted by $\tilde{\boldsymbol{u}}^{(t)}$, exhibits a \emph{gradient spike} if its deviation from the (unspiked) stochastic gradient evaluation satisfies
\[
\bigl\|\tilde{\boldsymbol{u}}^{(t)}-\boldsymbol{u}^{(t)}\bigr\|^2 = \mathcal{O}(1).
\]
By \eqref{error_variance} and \eqref{error_bias}, this occurs whenever there exists some $i\in\{1,2,\ldots,N-1\}$ such that either $\|\delta_i^{(t)}\|=\mathcal{O}(1)$ or $\|\varepsilon_i^{(t)}\|=\mathcal{O}(1)$.
Therefore, Table~\ref{table: error recover} directly implies how frequently gradient spikes may occur before they begin to slow down the convergence rate under the different assumptions.

\section{Experiments}
\label{section: Experiments}
In this section, we present numerical experiments to validate our theoretical 
findings. Specifically, we consider the logistic regression problem:
\begin{align}
    \min_{w\in\mathbb{R}^d}\ f(w):=\dfrac{1}{M}\sum_{l=1}^Mf_2(w;(h_l,y_l))+\rho\mathcal{R}(w).
\end{align}
Here, $\{(h_l, y_l)\}_{l=1}^{M}$ denote the training samples, where $h_l \in \mathbb{R}^d$ is the feature vector and $y_l \in \{+1,-1\}$ is the corresponding label. The logistic loss $f_2(w;(h_l,y_l))$ can be decomposed in the form of~\eqref{eq:pipeline-2} as:
\begin{equation}
\begin{aligned}
f_1(w;(h_l,y_l)):=-y_lh_l^\top w,\quad f_2(w;(h_l,y_l)):=\ln(1+\exp[f_1(w;(h_l,y_l))]).
\end{aligned} 
\end{equation}
The term $\mathcal{R}(w)$ is a regularizer and $\rho > 0$ is a given constant. 
We consider both a non-convex and a strongly convex choice of $\mathcal{R}$. 
In the non-convex case, following~\cite{antoniadis2011penalized,xin2021improved,
alghunaim2022unified,liang2023understanding}, we set $\mathcal{R}(w):=\sum_{j=1}^d{[w]_j^2}/{(1+[w]_j^2)}$,
where $[w]_j$ denotes the $j$-th entry of $w$. In the strongly convex case, 
we set $\mathcal{R}(w) = \|w\|_2^2$.
\begin{figure}[t!]
\centering
    \hspace{-5mm}
	\subfigure{
        \includegraphics[width=0.48\textwidth]{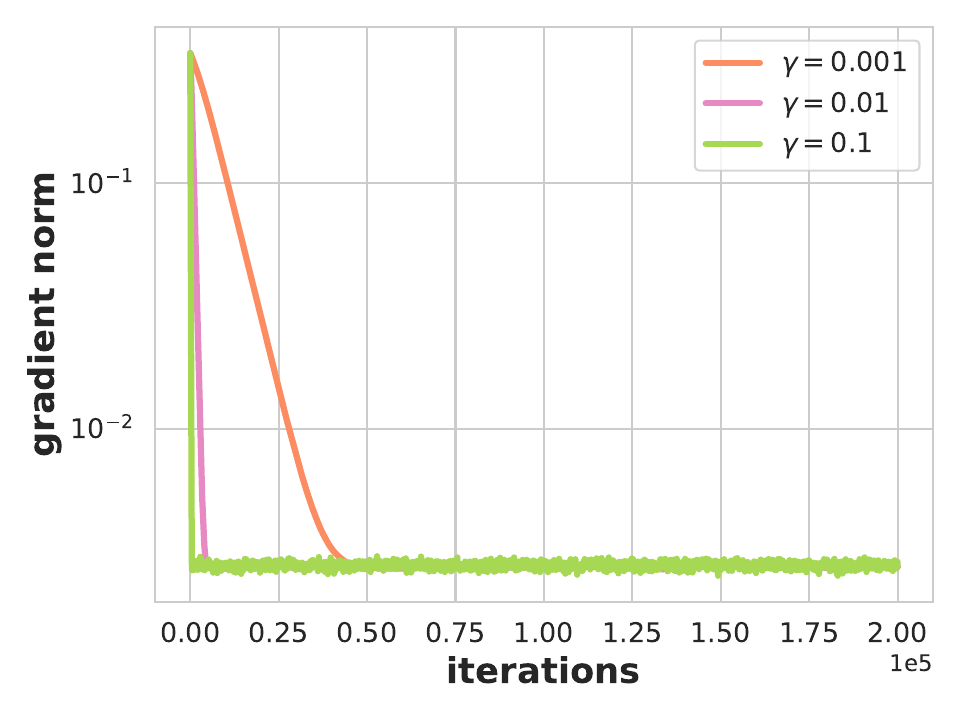}}
    \hspace{2mm}
	\subfigure{
		\includegraphics[width=0.48\textwidth]{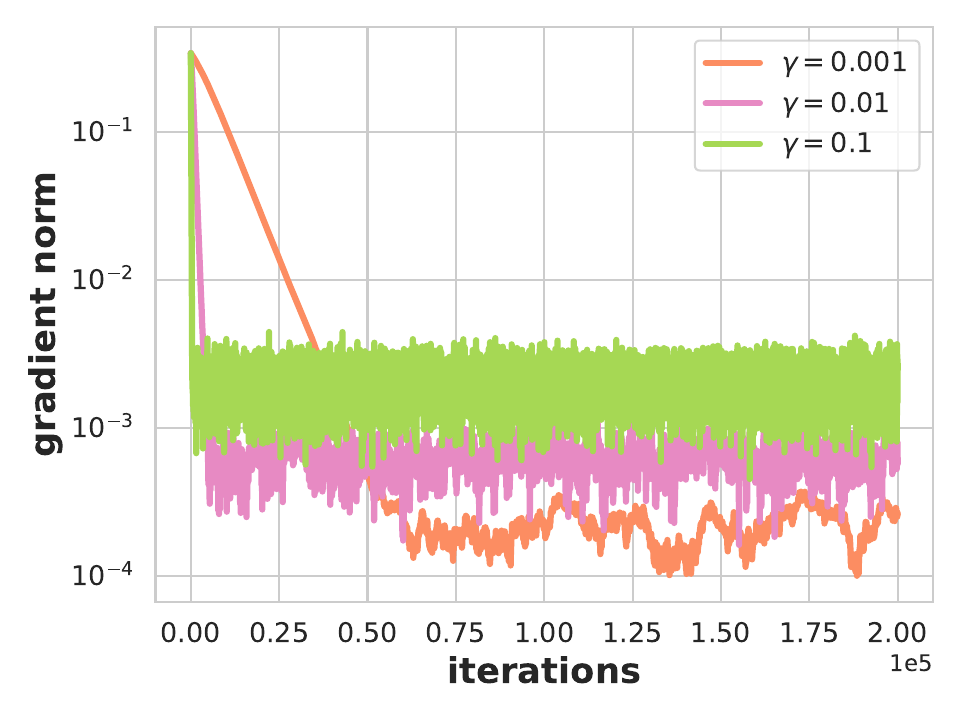}}
\caption{The convergence performance with forward and backward computation error with different step size $\gamma$ for the logistic regression task with non-convex regularization. (Left: $\sigma_f=2.0$, $\sigma_b=0.0$. Right: $\sigma_f=0.0$, $\sigma_b=2.0$.)}
\label{fig: nonconvex_unbiased_norm_show}
\end{figure}
\begin{figure}[t!]
\centering
    \hspace{-5mm}
	\subfigure{
        \includegraphics[width=0.48\textwidth]{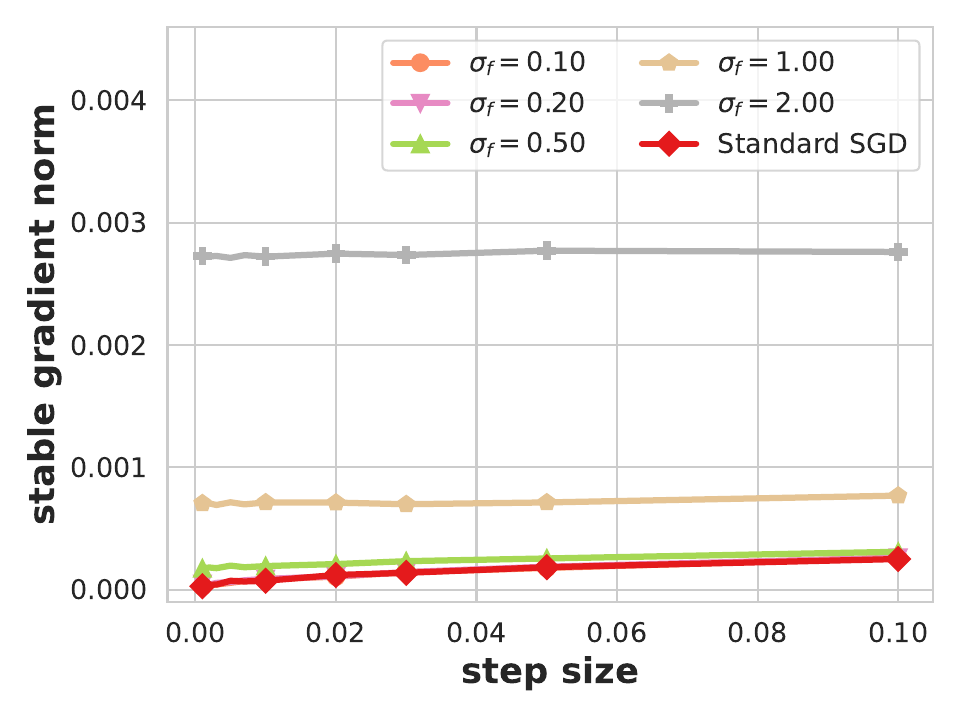}}
    \hspace{2mm}
	\subfigure{
		\includegraphics[width=0.48\textwidth]{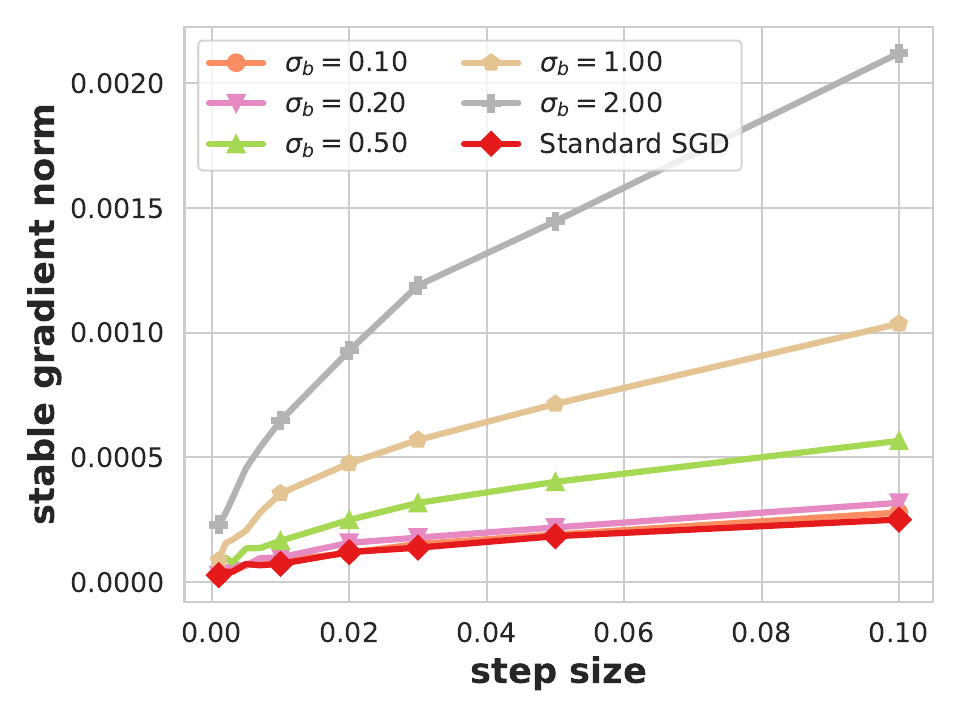}}
\caption{The relationship between the stable gradient norm and the step size $\gamma$ with forward and backward computation error for the logistic regression task with non-convex regularization. (Left: Varied $\sigma_f$ with $\sigma_b=0$. Right: Varied $\sigma_b$ with $\sigma_f=0$.)}
\label{fig: nonconvex_unbiased_iteration}
\end{figure}

\subsection{Experimental Setup}
We set $d = 10$, $M = 2000$, and use a fixed horizon $T=200{,}000$. To generate the data samples, we first draw $x^\star \sim \mathcal{N}(0, I_d)$ and feature vectors $h_l \sim \mathcal{N}(0, I_d)$. The corresponding label $y_l$ is obtained by drawing $z_l \sim \mathcal{U}(0,1)$ and setting $y_l = 1$ if $z_l \leq 1/(1+\exp(-h_l^{\top} x^\star))$ and $y_l = -1$ otherwise. At each gradient evaluation, we inject perturbations after computing the forward activation $f_1$ (forward perturbation $\delta_t$) and after computing the backward chain term $\partial f_2 / \partial f_1$ (backward perturbation $\varepsilon_t$). Let $\widetilde{\nabla f}(w_t)$ denote the resulting perturbed gradient. To emulate SGD with controlled variance, we form the stochastic gradient as
\[
g_t \;=\; \widetilde{\nabla f}(w_t) \;+\; \xi_t, \quad \xi_t \sim \mathcal{N}(0,\sigma_n^2 I_d), \quad \sigma_n = 0.001,
\]
where $\delta_t$ and $\varepsilon_t$ denote forward and backward computation perturbations, respectively, while $\xi_t$ denotes sampling noise.

To assess the effect of these perturbations on training, we use the gradient norm $\|\nabla f(w_t)\|$ as the primary evaluation metric. We smooth its trajectory via an exponentially weighted moving average (EWMA~\cite{hunter1986exponentially}) and define two derived quantities: the \emph{stable gradient norm}, the average EWMA value over the last $\lfloor T/2 \rfloor$ iterations; and the \emph{stable iteration complexity}, the first iteration at which the EWMA falls below $1.5\times$ the stable gradient norm.

As a baseline, we consider standard SGD with no perturbations in either pass (i.e., $\delta_t \equiv 0$ and $\varepsilon_t \equiv 0$ for all $t$), reducing the update to sampling noise only: 
\[ g_t = \nabla f(w_t) + \xi_t, \qquad w_{t+1} = w_t - \gamma g_t. \] 
This baseline appears as ``Standard SGD'' in the legend of plots where only one perturbation type is varied.

\subsection{Experimental Results}
In what follows, we report numerical results for three perturbation regimes (frequent zero-mean perturbations,
frequent non-zero-mean perturbations, and intermittent perturbations), and focus on a single observable throughout: whether the stable gradient norm continues to decrease toward the standard-SGD baseline as the step size shrinks, or instead saturates at a step-size-insensitive plateau. This behavior directly reflects the qualitative predictions of Section~\ref{subsection:continous error} for frequent perturbations (e.g., Remark~\ref{rmk-requrent}) and Section~\ref{subsection:intermittent error} for intermittent perturbations parameterized by occurrence budgets (e.g., Remark~\ref{rmk-intermittent}). Specifically, Section~\ref{section:7.2.1} considers frequent zero-mean perturbations and highlights the asymmetry between forward and backward perturbations; Section~\ref{section:7.2.2} isolates persistent non-zero-mean bias, which manifests as an irreducible plateau that cannot be eliminated by shrinking the step size alone; Section~\ref{section:7.2.3} studies intermittent perturbations, where varying the injection interval controls the effective occurrence budget and produces a clear phase transition.
\begin{figure}[t!]
\centering
    \hspace{-5mm}
	\subfigure{
        \includegraphics[width=0.48\textwidth]{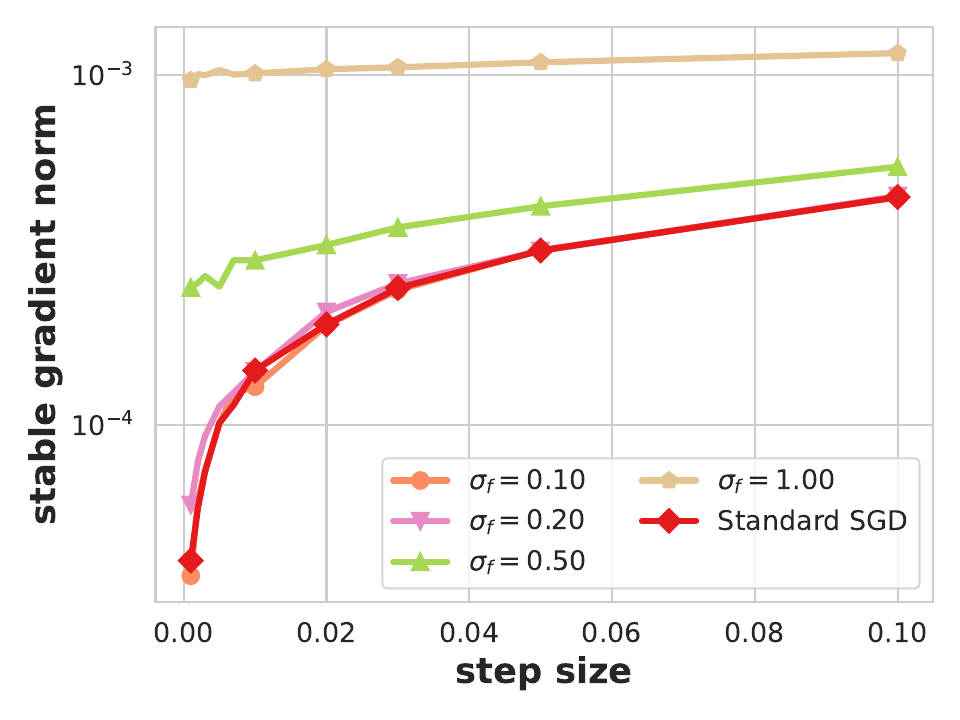}}
    \hspace{2mm}
	\subfigure{
		\includegraphics[width=0.48\textwidth]{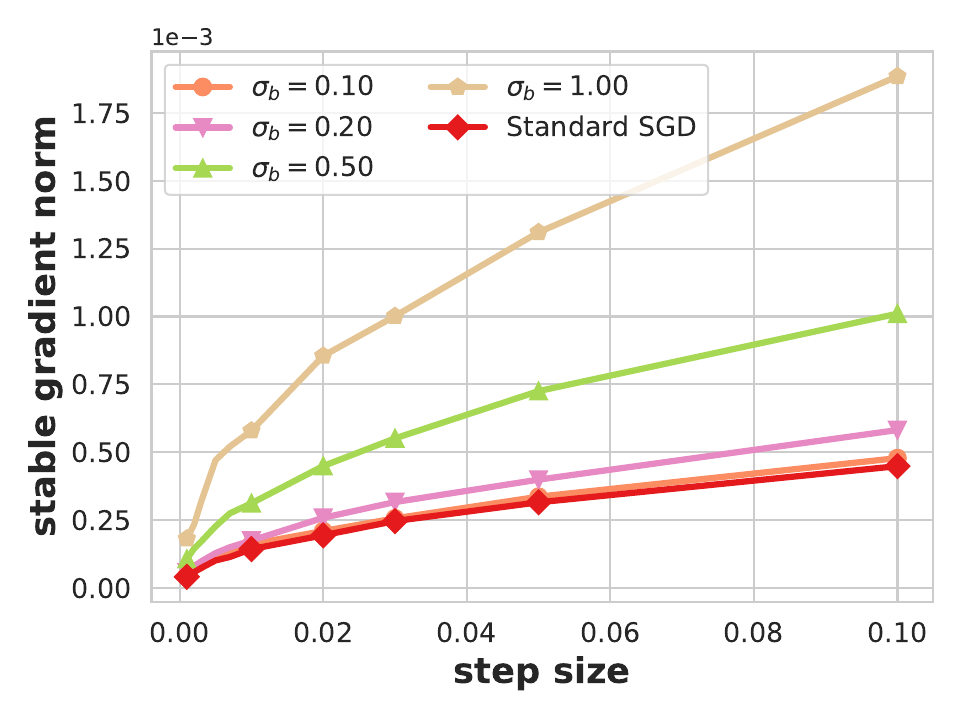}}
\caption{The relationship between the stable gradient norm and the step size $\gamma$ with forward and backward computation perturbations for the logistic regression task with strongly convex (hence PL) regularization. (Left: Varied $\sigma_f$ with $\sigma_b=0$. Right: Varied $\sigma_b$ with $\sigma_f=0$.)}
\label{fig: convex_unbiased_iteration}
\end{figure}

\subsubsection{Convergence with frequent zero-mean perturbations}
\label{section:7.2.1}
To validate convergence under zero-mean forward and backward perturbations, we inject symmetric noise $\delta$ in the forward pass and $\varepsilon$ in the backward pass, drawn independently and uniformly as $\delta\sim\mathcal{U}(-\sigma_f/\sqrt{d},\,\sigma_f/\sqrt{d})$ and $\varepsilon\sim\mathcal{U}(-\sigma_b/\sqrt{d},\,\sigma_b/\sqrt{d})$, where $\sigma_f$ and $\sigma_b$ control the respective perturbation magnitudes. The regularization coefficient is set to $\rho = 0.001$ in the nonconvex case and $\rho = 0.1$ in the strongly convex case.

In the nonconvex setting (Figures~\ref{fig: nonconvex_unbiased_norm_show} and~\ref{fig: nonconvex_unbiased_iteration}), the standard-SGD baseline exhibits the expected behavior: the stable gradient norm decreases and tends to $0$ as $\gamma\to 0$. Under zero-mean backward perturbations, the same qualitative behavior persists: Figure~\ref{fig: nonconvex_unbiased_norm_show} (right) shows smaller stabilized gradient norms for smaller step sizes, and Figure~\ref{fig: nonconvex_unbiased_iteration} (right) confirms that the stable gradient norm continues to decrease toward $0$ as $\gamma\to 0$ across all tested $\sigma_b$ values. 
This is consistent with Remark~\ref{rmk-requrent} (see also Remark~\ref{rem:freq-thresholds}): in the frequent unbiased regime, backward perturbations are comparatively benign and contribute primarily to the optimization error, so shrinking $\gamma$ reduces the stable gradient norm without inducing a step-size-insensitive plateau.

For zero-mean forward perturbations, however, the behavior depends on the perturbation magnitude. Figure~\ref{fig: nonconvex_unbiased_iteration} (left) indicates that for small perturbation levels the stable gradient norm still decreases with $\gamma$, but once the forward perturbation magnitude becomes non-negligible (e.g., $\sigma_f\ge 0.5$), the stable gradient norm saturates at a nonzero plateau even as $\gamma\to 0$. This plateau is further corroborated by the trajectories in Figure~\ref{fig: nonconvex_unbiased_norm_show} (left) under $\sigma_f=1.0$, where shrinking $\gamma$ yields little improvement. {This behavior illustrates the forward/backward asymmetry in Remark~\ref{rmk-requrent} (see also ~Remark~\ref{rem:freq-thresholds}): unless the forward perturbations satisfy the stricter decay conditions in Corollary~\ref{corollary:continous error}, the bound predicts a non-vanishing plateau even as $\gamma \to 0$.}

In the strongly convex (hence PL) setting (Figure~\ref{fig: convex_unbiased_iteration}), the same qualitative picture holds: the stable gradient norm under backward perturbations can still be driven toward $0$ as $\gamma\to 0$, whereas sufficiently large forward perturbations pin it at a nonzero plateau regardless of the step size. This is again consistent with Remark~\ref{rmk-requrent} together with the PL part of Corollary~\ref{corollary:continous error}: unbiased backward perturbations do not prevent the stable gradient norm from decreasing with $\gamma$, whereas sufficiently large forward perturbations pin it at a nonzero, step-size-insensitive plateau.

Overall, Figures~\ref{fig: nonconvex_unbiased_norm_show}--\ref{fig: convex_unbiased_iteration} provide a direct illustration of Remark~\ref{rmk-requrent}: decreasing the step size mainly reduces the optimization error component, whereas the perturbation-induced component manifests as a step-size-insensitive plateau when forward perturbations are too large. In line with Remark~\ref{rem:freq-thresholds} and Corollary~\ref{corollary:continous error}, the forward perturbations must satisfy stricter control conditions than the backward perturbations in order to avoid this plateau and recover the standard-SGD trend as $\gamma \to 0$.

\subsubsection{Convergence with frequent non-zero-mean computation perturbations}
\label{section:7.2.2}
\begin{wrapfigure}{r}{0.6\textwidth}
\centering
\vspace{-15pt}
\includegraphics[width=0.56\textwidth]{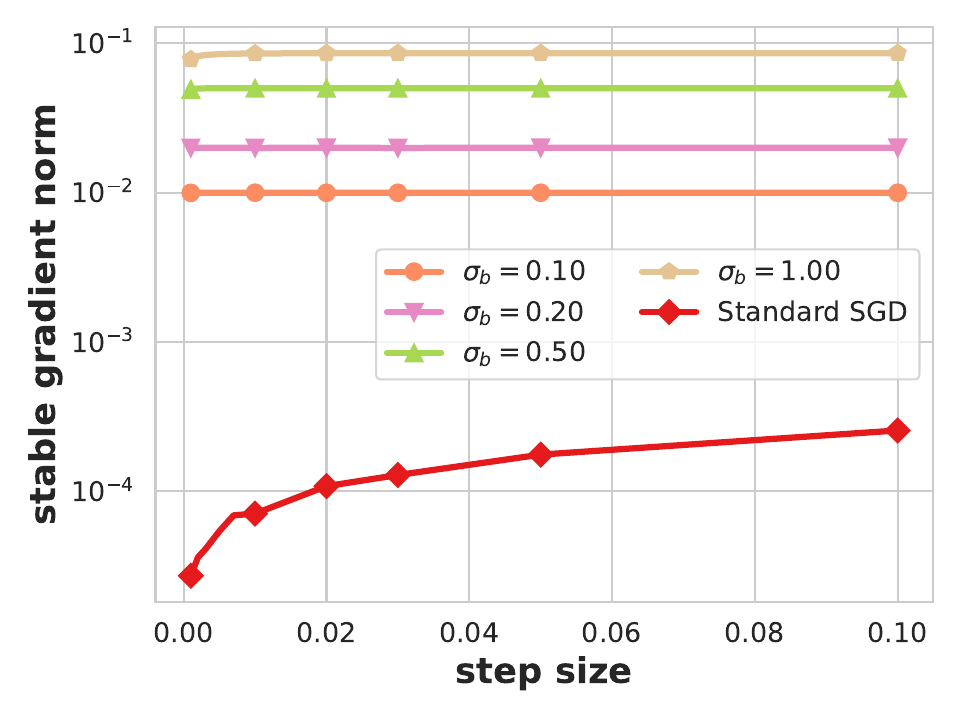}
\vspace{-10.5pt}
\caption{The relationship between the stable gradient norm and the step size $\gamma$ with non-zero-mean backward computation perturbations.}
 \vspace{-8pt}
\label{fig: nonconvex_biased_iteration}
\end{wrapfigure}
To isolate the effect of a persistent mean bias, we inject a non-zero-mean perturbation at every iteration in the \emph{backward} computation while keeping the forward pass unperturbed: $\varepsilon\sim\mathcal{U}(0,\sigma_b/\sqrt{d})$.
Figure~\ref{fig: nonconvex_biased_iteration} shows a clear qualitative difference from the zero-mean case:
even as $\gamma\to 0$, the stable gradient norm no longer tends to $0$ once $\sigma_b>0$, and instead saturates at a nonzero plateau that grows with $\sigma_b$.
The curve corresponding to zero perturbation magnitude (i.e., no injected computation perturbations) recovers the standard-SGD behavior, with the stable gradient norm decreasing toward $0$ as $\gamma\to 0$.

This is consistent with the conditional-bias requirements in Corollary~\ref{corollary:continous error} (see Remark~\ref{rem:freq-thresholds}) and the bias characterization in Theorem~\ref{thm-utilde-u}: once $\|\mathbb{E}_t^{\mathcal{H}_{i+1}}[\varepsilon_{i+1}^{(t)}]   \|$ is persistently nonzero (here induced by the non-zero-mean sampling), the theory contains an irreducible bias-driven term, which manifests empirically as a non-vanishing plateau that cannot be removed by shrinking $\gamma$ alone.

\subsubsection{Convergence with intermittent forward computation perturbations}
\label{section:7.2.3}
To investigate convergence properties under intermittent perturbations in forward propagation, we introduce zero-mean noise $\delta\sim\mathcal{U}(-\sigma_f/\sqrt{d},\sigma_f/\sqrt{d})$ during forward propagation at intervals of $\Delta t_f$ iterations in logistic regression with both non-convex and convex regularization terms. The parameter $\Delta t_f$ governs the noise injection frequency. The regularization coefficient $\rho$ is configured as 0.001 for non-convex scenarios and 0.5 for convex cases.

Figures~\ref{fig: unbiased_intermittent_nonconvex_iteration} and \ref{fig: unbiased_intermittent_convex_iteration} report the stable gradient norm versus the step size $\gamma$ under intermittent forward perturbations with injection interval $\Delta t_f$;
these correspond to the nonconvex and PL intermittent regimes analyzed in Corollary~\ref{corollary:spike_unbiased_nonconvex} and Corollary~\ref{corollary:spike_unbiased_pl}, respectively.
The curve labeled ``Standard SGD'' corresponds to standard SGD with no injected computation perturbations, and its stable gradient norm decreases toward $0$ as $\gamma\to 0$. When $\Delta t_f$ is small, the stable gradient norm exhibits a nonzero plateau even as $\gamma\to 0$, resembling the frequent-perturbation regime. As $\Delta t_f$ increases and spikes become rarer, the plateau height decreases and the curves progressively recover the error-free behavior, with the stable gradient norm again tending to $0$ as $\gamma\to 0$.
Moreover, for larger forward perturbation magnitudes, recovery requires a larger $\Delta t_f$, consistent with the fact that stronger spikes demand stricter sparsity.

Injecting a forward perturbation once every $\Delta t_f$ iterations implies that the number of forward spike iterations scales as $Q_\delta \approx \lceil T/\Delta t_f\rceil$ over a horizon $T$. 
Therefore, sweeping $\Delta t_f$ directly tunes the effective occurrence budget $Q_\delta$ appearing in Section~6.2 and Table~1.
In particular, Corollary~\ref{corollary:spike_unbiased_nonconvex} (nonconvex) and Corollary~\ref{corollary:spike_unbiased_pl} (PL) specify the admissible order of $Q_\delta$ (and $Q_\epsilon$) that preserves the error-free rate under zero-mean intermittent spikes,
and Remark~\ref{rmk-intermittent} interprets the resulting thresholds as a budget requirement; the progressive recovery in Figures~\ref{fig: unbiased_intermittent_nonconvex_iteration}--\ref{fig: unbiased_intermittent_convex_iteration} as $\Delta t_f$ increases is consistent with this budget-threshold picture.

\begin{figure}[t!]
\centering
    \hspace{-5mm}
	\subfigure{
        \includegraphics[width=0.48\textwidth]{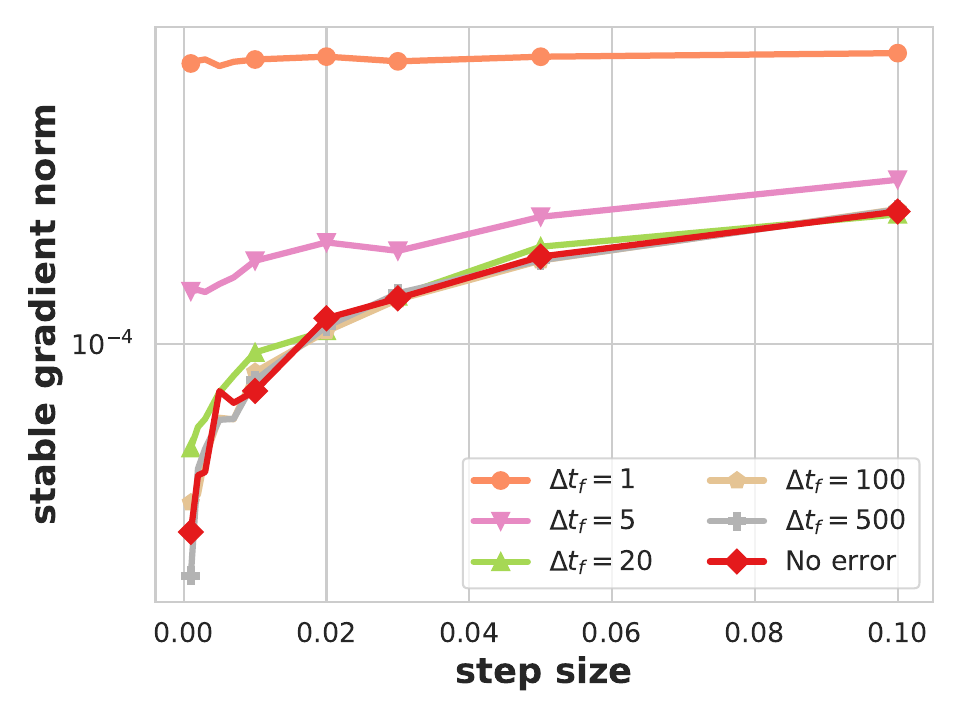}}
    \hspace{2mm}
	\subfigure{
\includegraphics[width=0.48\textwidth]{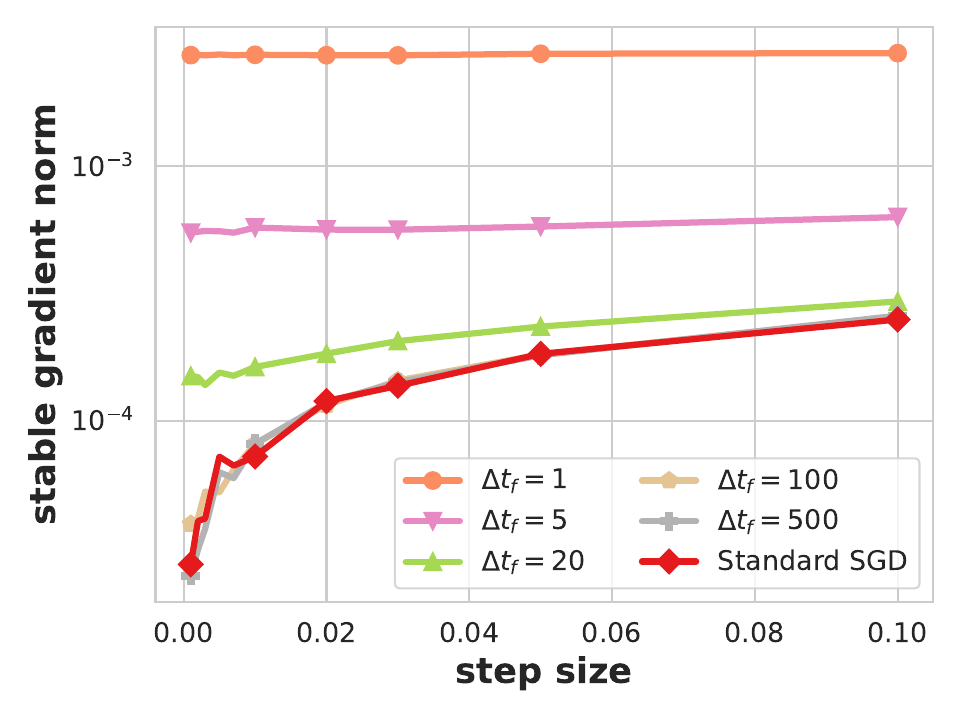}}
\caption{The relationship between the stable gradient norm and the step size $\gamma$ with intermittent forward perturbations for the logistic regression task in the non-convex scenario. (Left: $\sigma_f=1.0$. Right: $\sigma_f=2.0$.)}
\label{fig: unbiased_intermittent_nonconvex_iteration}
\end{figure}
\begin{figure}[t!]
\centering
    \hspace{-5mm}
	\subfigure{
        \includegraphics[width=0.48\textwidth]{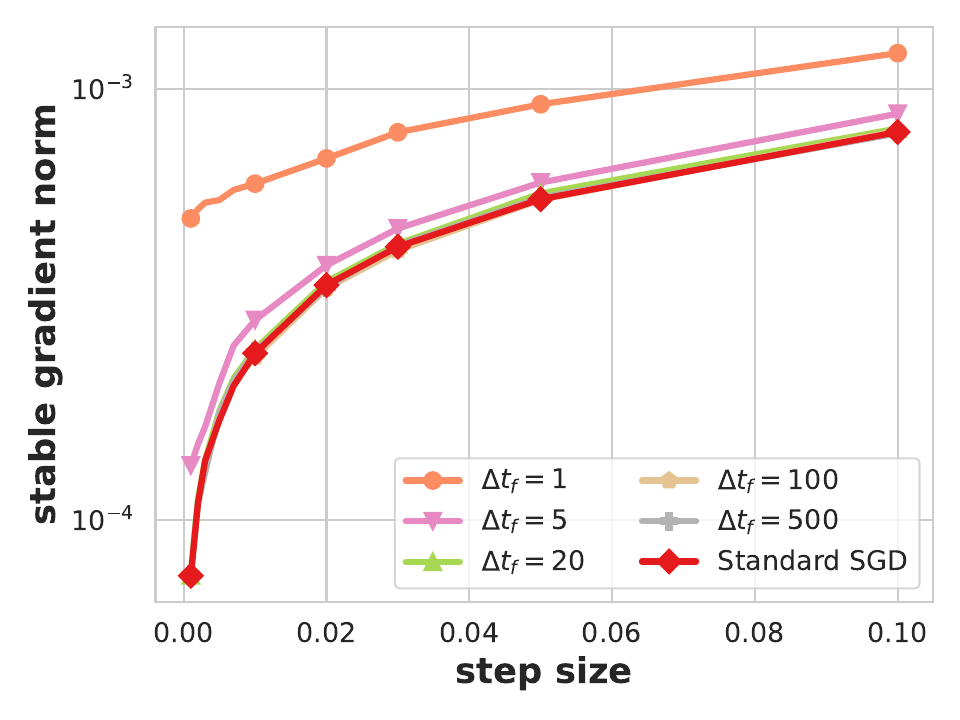}}
    \hspace{2mm}
	\subfigure{
        \includegraphics[width=0.48\textwidth]{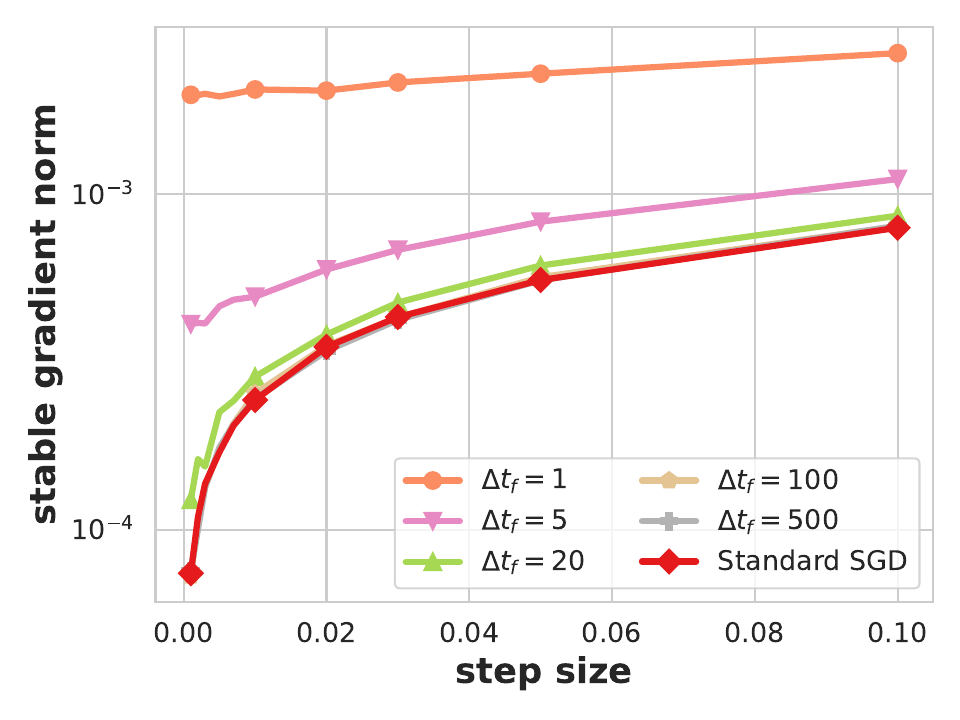}}
\caption{The relationship between the stable gradient norm and the step size $\gamma$ with intermittent forward perturbations for the logistic regression task in the strongly convex (hence PL) scenario. (Left: $\sigma_f=1.0$. Right: $\sigma_f=2.0$.)}
\label{fig: unbiased_intermittent_convex_iteration}
\end{figure}

\section{Conclusion and Future Works}
\label{section: Conclusion}
This paper studies the effect of computational perturbations on gradient propagation in stochastic gradient descent, with particular emphasis on their accumulation across successive operators. Through rigorous analysis of error dynamics in both the forward and backward passes, we establish convergence guarantees for SGD under general non-convex objectives and the Polyak--\L{}ojasiewicz condition in the presence of persistent computational inaccuracies. Our theoretical framework characterizes the admissible magnitudes for continuous perturbations and the tolerable frequencies for intermittent perturbations required to preserve standard convergence rates. Experimental results provide empirical validation of our theoretical propositions.

Future research directions encompass several key aspects. First, it remains to be seen whether simple correctors can restore error-free SGD rates under frequent $O(1)$ forward perturbations, for example through delta-based activation compression combined with error-compensation schemes that render accumulated distortion summable \cite{wang2022fine,richtarik2021ef21,fatkhullin2023momentum}. Second, explicit forward-pass de-biasing techniques, may eliminate the nonlinearity-induced bias terms underlying Theorem~\ref{thm-utilde-u}. Finally, extending the theory to deep and realistic mixed-precision or quantized pipelines, where perturbations are state-dependent and interact with momentum or adaptive stepsizes, presents a significant challenge.

\newpage
{
\small

\bibliography{reference}
\bibliographystyle{abbrv}
}

\newpage

\resettocdepth

\newpage
\appendix

\begin{center}
\Large
    \textbf{Appendix}
\end{center}

\section{Counter examples}
In this section, we present some counter examples that has not discussed in the previous sections.
\subsection{Zero-mean Noise can Lead to Biased Evaluated Gradient}
\label{section: example_zero-mean noise}
Here is a straightforward example that illustrates the theoretical challenge: even perturbations with zero expectation can introduce bias in the evaluated gradient.
%------------------------------------------------
% Biased Gradient Example
%------------------------------------------------
\begin{example}
Take the Sigmiod function $f_i(y) = \sigma(y) := (1 + e^{-y})^{-1}$, whose derivative given by $\sigma'(y) = \sigma(y)(1 - \sigma(y))$ is symmetric and reaches its peak when $y=0$. If we assume $y_{i-1}^{(t)} = 0$ and $\widetilde{y}_{i-1}^{(t)}$ follows a symmetric two-point distribution on $\{-a,a\}$ with equal masses. Consequently, $\mathbb{E}[\widetilde{y}_{i-1}^{(t)}] = 0 = y_{i-1}^{(t)}$, showing that $\widetilde{y}_{i-1}^{(t)}$ is an unbiased estimate of ${y}_{i-1}^{(t)}$. We find that $f_i'(y_{i-1}^{(t)}) = \sigma'(0)$. The expected value of the derivative calculates as:
\begin{align*}
\mathbb{E}[f_i'(\widetilde{y}_{i-1}^{(t)})] &= \frac{1}{2}\sigma'(a) + \frac{1}{2}\sigma'(-a)= \sigma'(a).
\end{align*}
Since $a>0$, it follows that $\sigma'(a) < \sigma'(0)$, leading to $\mathbb{E}[f_i'(\widetilde{y}_{i-1}^{(t)})] < f_i'(y_{i-1}^{(t)})$. Suppose $f_i'$ denotes the relevant gradient component of $\nabla f_i$ (like a partial derivative concerning the first variable), this inequality implies:
\begin{align}
\mathbb{E}[\nabla f_i (\tilde{y}_{i-1}^{(t)},w_i^{(t)})]\neq \nabla f_i ({y}_{i-1}^{(t)},w_i^{(t)}). \nonumber
\end{align}
\end{example}

This example highlights a key asymmetry between forward and backward perturbations: even when the injected noise is mean-zero at the level of intermediate variables, the resulting gradient estimate can become biased due to nonlinearity of the computational graph. This phenomenon motivates the higher-moment bias control terms (e.g., the $\mathbb{E}\|\delta\|^4$ contributions), and explains why the forward perturbation thresholds in Section~\ref{section: Recover from computation error and gradient spike} are typically more stringent than their backward counterparts.

\subsection{Non-convergence Induced by Frequent $\mathcal{O}(1)$ Magnitude Perturbations in Forward Propagation}
\label{section: A2}
We first construct a counterexample under gradient descent algorithm to demonstrate non-convergence caused by persistent $\mathcal{O}(1)$ magnitude perturbations.

Consider the deterministic gradient descent (GD) algorithm with objective function defined for all $x\in\mathbb{R}$ as:
\begin{align*}
f(x) = x^2.
\end{align*}
At iteration $t$, the forward propagation introduces an additive perturbation term yielding $\tilde{f}(x^{(t)})=(x^{(t)}+\delta)^2$, where $\delta\in\mathbb{R}$ represents a fixed perturbation constant. The corrupted gradient computation becomes $\boldsymbol{u}^{(t)} := 2(x^{(t)} + \delta)$. And this leads to the iteration scheme:
\begin{align}
x^{(t+1)} = x^{(t)} - \gamma\boldsymbol{u}^{(t)} =  x^{(t)} - 2\gamma(x^{(t)} + \delta) = (1 - 2\gamma)x^{(t)} - 2\gamma\delta,
\end{align}
where $\gamma$ denotes the learning rate.

For any $\gamma<1/2$, the sequence $\{x^{(t)}\}$ converges to $-\delta$ as $T\to+\infty$, demonstrating systematic deviation from the true minimum at zero. This phenomenon underscores the inherent instability of gradient-based methods when subjected to persistent $\mathcal{O}(1)$ forward propagation perturbations.

This example shows that persistent $\mathcal{O}(1)$-magnitude forward perturbations can shift the effective optimization dynamics, causing SGD to converge to a biased limit point (or, more generally, to a neighborhood that does not shrink with $T$). This concretely illustrates the ``error-floor'' interpretation of the bias-type terms in Theorem~\ref{thm: 5.2} and the necessity of the decay/sparsity conditions stated in Corollary~\ref{corollary:continous error}.

\subsection{Non-convergence With Top-K Compressor}
\label{section: A3}
We construct a counter example to illustrate the non-convergence of SGD algorithm with standard non-convex assumption under Top-K compression. 

Specifically, we denote the sample and parameter vectors as:
\begin{align*}x = \left( x_1,x_2 \right)^\top,\quad w_1 = \left( w_{11},w_{12} \right)^\top,\quad w_2 = \left( w_{21},w_{22} \right)^\top, \quad A = \begin{pmatrix}
    a &b \\
    c &d
\end{pmatrix}.
\end{align*} 
Then, we take the objective function $f(\boldsymbol{x},\w):=f_2(f_1(\boldsymbol{x},\w_1), \w_2)$, where the function $f_1$ and $f_2$ are defined as:
\begin{align*}f_1(x, w_1) = A (x \odot w_1) =\begin{pmatrix} a x_1 w_{11} + b x_2 w_{12}\\ c x_1 w_{11} + d x_2 w_{12}  \end{pmatrix},\quad f_2(y_1,w_2) = \langle y_1,w_2 \rangle^2 + \frac{1}{2} \| y_1 \|^2.\end{align*}

Then, we introduce the Top-1 compression operator $\mathcal{C}_G$, which select the element with the maximum absolute value and set all the other elements to zero. With $\mathcal{C}_G$, the gradients of $f_1$ and $f_2$ can be computed as follows:
\begin{equation}
\begin{aligned}
    u_2 &:= \nabla_{w_2} f_2(y_1,w_2) = 2 \langle {y}_1, w_2 \rangle {y}_1,\\
    {v}_1 &:= \nabla_{{y}_1} f_2({y}_1,w_2) = 2 \langle {y}_1, w_2 \rangle {w}_2 + {y}_1 \\
    {u}_1 &:= \left(\nabla_{{w}_1} f_1({x},{w}_1)\right)^{\top} {\mathcal{C}_G} \left({v}_1\right) = \left( \begin{array}{cc}
             a x_1 &b x_2  \\
             c x_1 &d x_2
        \end{array} \right)^{\top} {\mathcal{C}_G} \left(2 \langle {y}_1, {w}_2 \rangle {w}_2 + {y}_1\right),
\end{aligned}
\end{equation}
where $\mathcal{C}_G$ denotes the Top-1 compression operator.

\begin{figure}[t!]
\centering
    \hspace{-5mm}
	\subfigure{
        \includegraphics[width=0.48\textwidth]{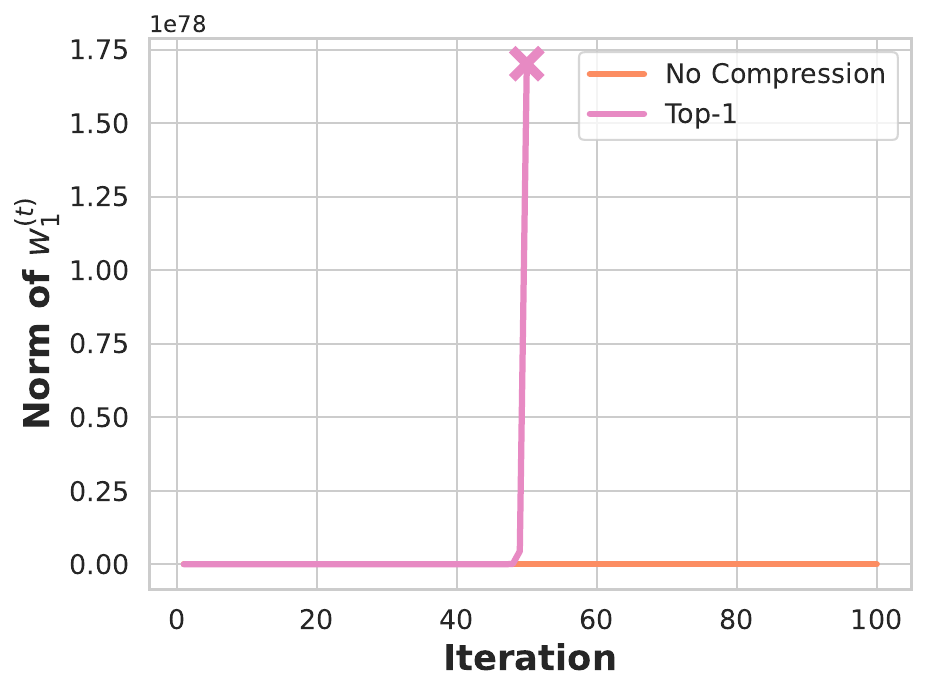}}
    \hspace{2mm}
	\subfigure{
        \includegraphics[width=0.48\textwidth]{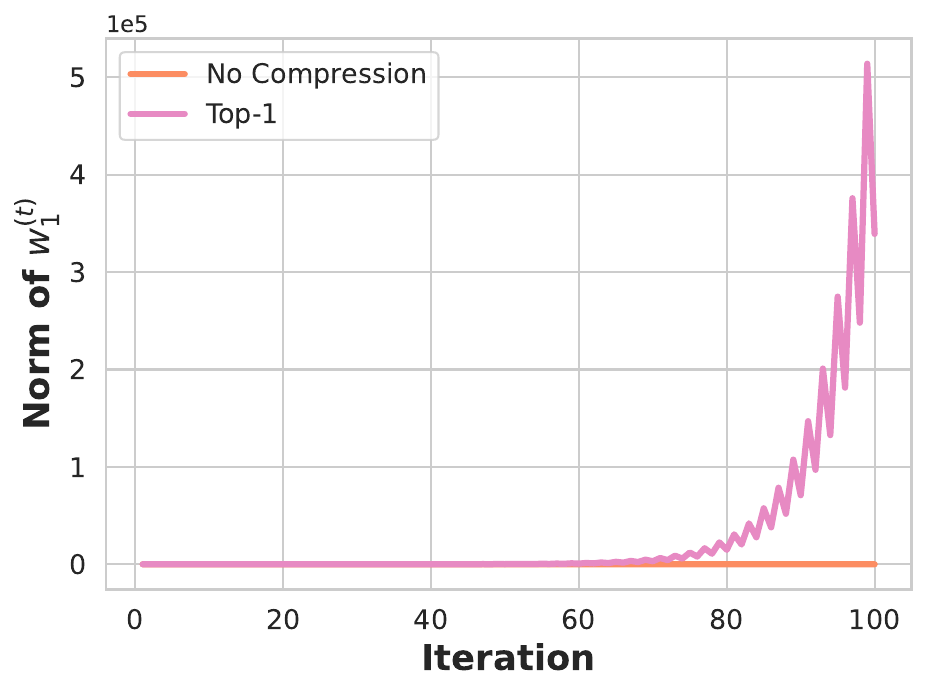}}
\caption{The norm of the parameter $w_1^{(t)}$ obtained by SGD with Top-1 compressor and the comparison with the non-compression case. Note that the Top-1 cases all fail to converge. (Left: $\gamma=0.1$ and $\beta=0$. Right: $\gamma=0.05$ and $\beta=0.9$.)}
\label{fig: non_convergence_top1}
\end{figure}

Set the initial values ${w}_1^{(1)} = \left( s, s  \right)^{\top}$ and ${w}_2^{(1)} = \left(  0, 0 \right)^{\top}$, and fix ${x}^{(k)} \equiv x = \left( 1, 1 \right)^{\top}$ for all $t = 1,2,3,\cdots$, where $s$ is a given coefficient. Then we can derive the recurrence relations for the $t$-th iteration as follows:
\begin{equation*}
\begin{aligned}
{y}_1^{(t)} &= \left( \begin{array}{cc} a &b \\ c &d \end{array} \right) {w}_1^{(t)}, \quad {v}_1^{(t)} = {\mathcal{C}_G}\left( {y}_1^{(t)} \right), \quad {u}_1^{(t)} = \left( \begin{array}{cc} a &b \\ c &d \end{array} \right)^{\top} {v}_1^{(t)}, \\
{w}_1^{(t+1)} &= {w}_1^{(t)} - \gamma {u}_1^{(t)}, \quad {u}_2^{(t)} = {0}, \quad {w}_2^{(t)} \equiv \left( \begin{array}{c} 0 \\ 0 \end{array} \right).
\end{aligned}
\end{equation*}
If we take $a = -15, b = 13, c = -5, d = 9$ and set the learning rate $\gamma = 0.1$, our counterexample converges to ${w}_1 = {0}$ without compression (i.e., when $\mathcal{C}_G$ is the identity operator). Specifically, with no computation perturbation, we have:
\begin{align*}
{y}_1^{(1)} = \left( \begin{array}{c} -2s \\ 4s \end{array} \right), \quad {v}_1^{(1)} = \left( \begin{array}{c} -2s \\ 4s \end{array} \right), \quad {u}_1^{(1)} = \left( \begin{array}{c} 10s \\ 10s \end{array} \right), \quad {w}_1^{(2)} = \left( \begin{array}{c} 0 \\ 0 \end{array} \right)
\end{align*}

With momentum SGD, the updating rule of ${w}_1^{(t)}$ and ${w}_2^{(t)}$ becomes:
\[
    {w}_1^{(t)} = {w}_1^{(t-1)} - \eta {u}_1^{(t)} + \beta \left( {w}_1^{(t-1)} - {w}_1^{(t-2)} \right)
\]
Hence, when we take $a = -5, b = -4, c = -5, d = 3, \gamma = 0.05 ,\beta = 0.9$, numerical study implies that the momentum SGD without the compression error achieve the convergence as Figure \ref{fig: non_convergence_top1}. 

However, Figure \ref{fig: non_convergence_top1} also illustrates that the parameter cannot convergence with the Top-1 compression with the same setup as the no-compression cases in both SGD optimization and momentum SGD. This result validates our theoretical finding that it cannot converge under frequent perturbation with $\mathcal{O}(1)$ magnitude.

\end{document}